\documentclass[12pt,reqno]{amsart}
\usepackage{times}
\usepackage[T1]{fontenc}
\usepackage{mathrsfs}
\usepackage{latexsym}
\usepackage[dvips]{graphics}
\usepackage{epsfig}
\usepackage{hyperref, amsmath, amsthm, amsfonts, amscd, flafter,epsf}
\usepackage{amsmath,amsfonts,amsthm,amssymb,amscd}
\input amssym.def
\input amssym.tex
\usepackage{color}
\usepackage{enumerate}
\usepackage{hyperref}
\usepackage{url}
\usepackage{floatrow}
\usepackage{caption}
\usepackage{subcaption}
\usepackage{capt-of}

\newcommand\be{\begin{equation}}
\newcommand\ee{\end{equation}}
\newcommand\bea{\begin{eqnarray}}
\newcommand\eea{\end{eqnarray}}
\newcommand\bi{\begin{itemize}}
\newcommand\ei{\end{itemize}}
\newcommand\ben{\begin{enumerate}}
\newcommand\een{\end{enumerate}}

\newtheorem{thm}{Theorem}[section]

\newtheorem{lem}[thm]{Lemma}

\newcommand{\burl}[1]{\textcolor{blue}{\url{#1}}}
\numberwithin{equation}{section}
\begin{document}
\title[Random Matrix Ensembles with a Polynomial Link Function]{Limiting Spectral Measures for Random Matrix Ensembles with a Polynomial Link Function}
\author{Kirk Swanson}
\email{swansonkirk80@gmail.com}
\address{Department of Mathematics and Statistics, Williams College, Williamstown, MA 01267}
\author{Steven J. Miller}
\email{sjm1@williams.edu, Steven.Miller.MC.96@aya.yale.edu} \address{Department of Mathematics and Statistics, Williams College, Williamstown, MA 01267}
\author{Kimsy Tor}\email{ktor.student@manhattan.edu}
\address{Department of Mathematics, Manhattan College, Bronx, NY 10463}
\author{Karl Winsor}
\email{krlwnsr@umich.edu}
\address{Department of Mathematics, University of Michigan, Ann Arbor, MI 48109}
\date{\today}
\thanks{This work was supported by NSF grants DMS1347804 and DMS1265673, and Williams College. We thank Arup Bose and Fred Strauch for many helpful discussions.}
\begin{abstract}
Consider the ensembles of real symmetric Toeplitz matrices and real symmetric Hankel matrices whose entries are i.i.d. random variables chosen from a fixed probability distribution \textit{p} of mean 0, variance 1, and finite higher moments.  Previous work on real symmetric Toeplitz matrices shows that the spectral measures, or densities of normalized eigenvalues, converge almost surely to a universal near-Gaussian distribution, while previous work on real symmetric Hankel matrices shows that the spectral measures converge almost surely to a universal non-unimodal distribution.  Real symmetric Toeplitz matrices are constant along the diagonals, while real symmetric Hankel matrices are constant along the skew diagonals.  We generalize the Toeplitz and Hankel matrices to study matrices that are constant along some curve described by a real-valued bivariate polynomial.  Using the Method of Moments and an analysis of the resulting Diophantine equations, we show that the spectral measures associated with linear bivariate polynomials converge in probability and almost surely to universal non-semicircular distributions.  We prove that these limiting distributions approach the semicircle in the limit of large values of the polynomial coefficients.  We then prove that the spectral measures associated with the sum or difference of any two real-valued polynomials with different degrees converge in probability and almost surely to a universal semicircular distribution.
\end{abstract}
\maketitle
%%%%%%%%%%%%%%%%%%%%%%%%%%%%%%%%%%%%%%%%%%%%%%%%%%%
\normalsize
%%%%%%%%%%%%%%%%%%%%%%%%%%%%%%%%%%%%%%%%%%%%%%%%%%%%%%%%%%%%%%%%%%%%%%%%%%%%%%%%%%%%%%%%%%%%%%%%%%%%%%%%%%%%%%%%%%%%%%%%%%%%%%
%%%%%%%%%%%%%%%%%%%%%%%%%%%%%%%%%%%%%%%%%%%%%%%%%%%%%%%%%%%%%%%%%%%%%%%%%%%%%%%%%%%%%%%%%%%%%%%%%%%%%%%%%%%%%%%%%%%%%%%%%%%%%%
%%%%%%%%%%%%%%%%%%%%%%%%%%%%%%%%%%%%%%%%%%%%%%%%%%%%%%%%%%%%%%%%%%%%%%%%%%%%%%%%%%%%%%%%%%%%%%%%%%%%%%%%%%%%%%%%%%%%%%%%%%%%%%
\section{Introduction}\label{sec:sec1}
\subsection{Background}\label{sec:sec11}
Ever since Eugene Wigner conjectured that random matrices could be used to approximate the spacing distribution between adjacent energy levels in heavy nuclei (see \cite{Wig1, Wig2, Wig3, Wig4, Wig5}), random matrix theory has been a powerful tool in modeling many complex systems, as described in Firk and Miller \cite{FM} and exemplified in Miller, Novikoff, and Sabelli \cite{MNS}, Baik, Borodin, Deift, and Suidan \cite{BBDS}, and Krbalek and Seba \cite{KrSe}, to name a few diverse representatives.  For example, while the Schr{\"o}dinger equation can be solved for the simple nuclear structure of hydrogen, there is no known closed-form expression for the nuclear energy levels of uranium, which has over 200 protons and neutrons interacting in its nucleus.   Wigner's great insight was to model the infinite dimensional Hamiltonian matrix with the limit of $\textit{N}\times\textit{N}$ real symmetric matrices, in which each entry is chosen randomly from a Gaussian density.  For each \textit{N}, one calculates averages over a weighted set of all possible matrices of size \textit{N}, such as the average density of normalized eigenvalues (analagous to the energy levels).  Similar to the Central Limit Theorem, as $\textit{N}\rightarrow\infty$ the behavior of the normalized eigenvalues of almost any randomly chosen matrix agrees with the limits of the system averages.
Although Wigner chose the Gaussian density, one could instead choose a generic probability distribution \textit{p} with mean 0, variance 1, and finite higher moments.  For example, for real symmetric matrices with entries chosen as i.i.d. random variables from such a probability distribution, it has been proved that the limiting distribution of the density of normalized eigenvalues is the semicircle law (see \cite{Wig2,Wig6}) independent of \textit{p}.  Although there are interesting results that have been proved for the adjacent spacing distribution of normalized eigenvalues that only recently have been proved to hold for all such $p$ (see for example \cite{ERSY, ESY, TV1, TV2}), we will focus only on densities in this paper.
It is interesting to explore how the eigenvalue behavior changes when additional structure is imposed on real symmetric matrices.  Real symmetric matrices have $\frac{N(N+1)}{2}$ independent parameters, from the upper triangle of the matrix.  For sub-ensembles that have fewer degrees of freedom, different limiting distributions can arise.  One example of a thin subset that only has \textit{N} degrees of freedom is the $\textit{N}\times\textit{N}$ real symmetric Toeplitz matrices, which are constant along the diagonals:
\begin{align}\label{1.1} T_N\ = \ \left( \begin{array}{ccccc}
a_{0} & a_{1}& a_{2}& \dots & a_{N-1}\\
a_{1} & a_{0}& a_{1}& \dots & a_{N-2}\\
a_{2} & a_{1}& a_{0}& \dots & a_{N-3}\\
\vdots & \vdots & \vdots & \ddots & \vdots \\
a_{N-1} & a_{N-2}& a_{N-3}& \dots & a_{0} \end{array} \right). \end{align}
Another is the real symmetric Hankel matrices, which are constant along the skew diagonals:
\begin{align}\label{1.2} H_N\ = \ \left( \begin{array}{ccccc}
a_{2} & a_{3}& a_{4}& \dots & a_{N+1}\\
a_{3} & a_{4}& a_{5}& \dots & a_{N+2}\\
a_{4} & a_{5}& a_{6}& \dots & a_{N+3}\\
\vdots & \vdots & \vdots & \ddots & \vdots \\
a_{N+1} & a_{N+2}& a_{N+3}& \dots & a_{2N} \end{array} \right). \end{align}
Bai \cite{Bai} proposed studying the density of normalized eigenvalues of real symmetric Toeplitz
matrices.  Initially, numerical investigations suggested that the density of normalized eigenvalues might converge to the Gaussian; however, Bose, Chatterjee, and Gangopadhyay \cite{BCG}, Bryc, Dembo, and Jiang \cite{BDJ}, and Hammond and Miller \cite{HM} showed that this is not the case by calculating the fourth moment to be 2 $\frac{2}{3}$, close to but not equal to the standard Gaussian's fourth moment of 3.  Massey, Miller, and Sinsheimer \cite{MMS} then found that by imposing additional structure on the Toeplitz matrices by making the first row a palindrome, the limiting spectral measures converge in probability and almost surely to the standard Gaussian.  Other generalizations include studying the effect of increasing the palindromicity of palindromic Toeplitz matrices (see \cite {JMP}) and scaling each entry by the square root of the number of times that entry appears in the matrix (see \cite{BB}).  In this paper we explore another generalization of the real symmetric Toeplitz and Hankel matrices by studying matrices that are constant along some curve described by a real-valued bivariate polynomial.  We begin by listing our notation below and then stating our results in \S\ref{sec:sec13}.
%%%%%%%%%%%%%%%%%%%%%%%%%%%%%%%%%%%%%%%%%%%%%%%%%%%%%%%%%%%%%%%%%%%%%%%%%%%%%%%%%%%%%%%%%%%%%%%%%%%%%%%%%%%%%%%%%%%%%%%%%%%%%%
%%%%%%%%%%%%%%%%%%%%%%%%%%%%%%%%%%%%%%%%%%%%%%%%%%%%%%%%%%%%%%%%%%%%%%%%%%%%%%%%%%%%%%%%%%%%%%%%%%%%%%%%%%%%%%%%%%%%%%%%%%%%%%
\subsection{Notation}\label{sec:sec12}
\subsubsection{Random Matrices and Link Functions.}
A $\textit{random matrix}$ $A_N$ is an $N \times N$ matrix whose entries are random variables drawn from a fixed probability distribution \textit{p} with mean 0, variance 1, and finite higher moments.  A particular random matrix is constructed from a sequence of i.i.d. random variables with distribution \textit{p} called the $\textit{input sequence}$: $\{a_i:i \in \mathbb{Z}\}$.  The pattern of a random matrix is determined by the $\textit{link function}$, $L(i,j)$, which maps the entry locations of a matrix $(i,j)$ to the input sequence of random variables.\footnote{We follow the expositions in \cite{B,HM,X} in the remaining introductory sections.  }  Real symmetric Toeplitz matrices, for example, have the link function  $L_{\rm Toeplitz}(i,j) = |i-j|$.  The probability that the entry at $(i,j)$ lies in the interval $[\alpha_{ij},\beta_{ij}]$ for a matrix $A_N$ contained in the \textit{outcome space} $\Omega_N$ is given by
\begin{equation}\label{1.3}
\text{Prob}\left(A_N\in\Omega_N: a_{L(i,j)}\in[\alpha_{ij},\beta_{ij}]\right)\ =\ \prod_{1\le i \le j \le N}\int_{\alpha_{ij}}^{\beta_{ij}}p(x)dx.
\end{equation}
The probability measure for the normalized eigenvalues of a real symmetric matrix $A_N$, called the \textit{empirical spectral measure}, is denoted by\footnote{The $N$ real eigenvalues are ordered as $\lambda_1\le\lambda_2\le\cdots\le\lambda_N$.}
\begin{align}\label{1.4}
\begin{split}
\mu_{A_N}(x)dx\ \ := \ \ \frac{1}{N}\sum_{i=1}^N\delta\left(x-\frac{\lambda_i(A_N)}{\sqrt{N}}\right)dx,
\end{split}
\end{align}
where $\delta(x)$ is the Dirac-delta functional.\footnote{See \cite{MT} for a heuristic for the eigenvalues of our $N\times N$ matrix ensembles being roughly of size $\sqrt{N}$.  Although we have chosen to keep $\sqrt{N}$ as the normalization, it would be just as effective to scale them by $2\sqrt{N}$.  For a scaling of $\sqrt{N}$, real symmetric matrices have a semicircular limiting spectral distribution given by $\frac{1}{2\pi}\sqrt{4-x^2}$  for $|x| \le 2$ and 0 otherwise, whose moments are exactly the Catalan numbers.  For a scaling of 2$\sqrt{N}$, the distribution is $\frac{2}{\pi}\sqrt{1-x^2}$ for $|x|\le1$ and 0 otherwise, whose moments are proportional to the Catalan numbers.\label{foot:foot3}}  We can use this to define the $\textit{empirical spectral}\\\textit{distribution}$, a cumulative distribution function, for a matrix of size $N$:
\begin{align}
\begin{split}
F_{A_N}(x) \ := \ \int_{-\infty}^x\mu_{A_N}(x)dx\ = \ \frac{\#\left\{ i \le N : \frac{\lambda_i}{\sqrt{N}}\le x\right\}}{N}.
\end{split}
\end{align}
%%%%%%%%%%%%%%%%%%%%%%%%%%%%%%%%%%%%%%%%%%%%%%%%%%%%%%%%%%%%%%%%%%%%%%%%%%%%%%%%%%%%%%%%%%%%%%%%%%%%%%%%%%%%%%%%%%%%%%%%%%%%%%
%%%%%%%%%%%%%%%%%%%%%%%%%%%%%%%%%%%%%%%%%%%%%%%%%%%%%%%%%%%%%%%%%%%%%%%%%%%%%%%%%%%%%%%%%%%%%%%%%%%%%%%%%%%%%%%%%%%%%%%%%%%%%%
\subsubsection{The Method of Moments.}
We want to understand the distribution of eigenvalues in the limit as the size of the matrices grows to infinity, which amounts to understanding convergence properties of the cumulative distribution functions of typical random matrices.  The critical connection is established in the following moment convergence theorem.\footnote{See \cite{B} for more details.}
\begin{thm}[The Method of Moments]\label{thm:thm11} Let $\{A_N\}_{N=1}^{\infty}$ be a sequence of random variables and $\{F_{N}\}_{N=1}^{\infty}$ be the corresponding sequence of cumulative distribution functions such that their moments, $M_{k}(N) = \int_{-\infty}^{\infty}x^kdF_{N}(x)$, exist for all positive integers $k$.  Let $\{M_{k}\}_{k=1}^{\infty}$ be a sequence of moments that uniquely determine a probability distribution\footnote{Let $\{M_k\}_{k=1}^{\infty}$ be the sequence of moments for the limiting spectral distribution $F$.  Then, $F$ is the unique distribution with these moments if $\text{lim}_{k\rightarrow\infty}\text{inf}\frac{1}{k}M_{2k}^{\frac{1}{2k}}<\infty$.  This is called \textit{Riesz's Condition} (see \cite{B}).} whose cumulative distribution function is denoted by $F$.  If $\text{lim}_{N\rightarrow \infty}M_{k}(N) = M_k$ for each positive integer $k$, then the sequence of cumulative distribution functions for the random variables converges weakly to the limiting distribution: $\text{lim}_{N \rightarrow \infty}F_N = F$. \end{thm}
The $k\textsuperscript{th}$ moment for the empirical spectral distribution of a random matrix $A_N$ is
\begin{align}\label{1.6}
\begin{split}
M_{k}(A_N)\ :=\ \int_{-\infty}^{\infty}x^k\mu_{A_N}(x)dx\ =\ \frac{1}{N^{\frac{k}{2}+1}}\sum_{i=1}^{N}\lambda_i^k(A_N).
\end{split}
\end{align}
To explore the behavior of the limiting spectral distribution for a typical sequence of random matrices, we compute the average moment values over all such matrices where, for a given \textit{N}, the average $k\textsuperscript{th}$ moment for matrices of size $N$ weighted by Eq. \eqref{1.3} is
\begin{align}
\begin{split}
M_{k}(N)\ :=\ \mathbb{E}\left[M_{k}(A_N)\right],
\end{split}
\end{align}
and the $k\textsuperscript{th}$ moment of the limiting spectral distribution is
\begin{align}
\begin{split}
M_{k} \ := \  \text{lim}_{N\rightarrow \infty}M_{k}(N).
\end{split}
\end{align}
The Moment Convergence Theorem for Random Matrices, which follows directly from the Method of Moments, serves as our main tool for understanding the limiting distribution.\footnote{See \cite{B} for more details.}
\begin{thm}[Moment Convergence Theorem for Random Matrices]\label{thm:thm12} Suppose $\{A_N\}_{N=1}^{\infty}$ is an arbitrary sequence of random matrices with distributions $\{F_{A_N}\}_{N=1}^{\infty}$.  Suppose there exists some sequence of moments $\{M_{k}\}_{k=1}^{\infty}$ such that they uniquely determine a probability distribution whose cumulative distribution function is denoted by $F$. If $\text{ lim}_{N\rightarrow \infty}M_{k}(N)= M_{k} \text{ and } \text{lim}_{N\rightarrow \infty}\text{Var}[M_{k}(A_N)]= 0$ for every positive integer k, then the sequence $\{F_{A_N}\}_{N=1}^{\infty}$ converges in probability to the limiting spectral distribution of the ensemble, $F$.\end{thm}
%%%%%%%%%%%%%%%%%%%%%%%%%%%%%%%%%%%%%%%%%%%%%%%%%%%%%%%%%%%%%%%%%%%%%%%%%%%%%%%%%%%%%%%%%%%%%%%%%%%%%%%%%%%%%%%%%%%%%%%%%%%%%%
%%%%%%%%%%%%%%%%%%%%%%%%%%%%%%%%%%%%%%%%%%%%%%%%%%%%%%%%%%%%%%%%%%%%%%%%%%%%%%%%%%%%%%%%%%%%%%%%%%%%%%%%%%%%%%%%%%%%%%%%%%%%%%
\subsubsection{Circuits.}  Using the Eigenvalue-Trace Lemma, we have\footnote{Although $a_{ij}$ refers to a matrix entry and $a_{L(i,j)}$ refers to the input sequence variable whose value describes that entry, it is convenient to use both notations interchangeably.}
\begin{align}\label{1.9}
\begin{split}
M_{k}(N)=\frac{1}{N^{\frac{k}{2}+1}}\sum_{1\leq i_1,\dots,i_k\leq N}\mathbb{E}[a_{L(i_1,i_2)}a_{L(i_2,i_3)}\cdot \cdot \cdot a_{L(i_k,i_1)}].
\end{split}
\end{align}
The above sum is taken over all combinations of positive integers $\{i_1,\dots,i_k\}$ at most $N$, where each distinct combination, or \textit{circuit}, is a function from indices to integer values:\footnote{Here the index values have been shifted down by one, so that the first index $i_1$ is mapped by  $\pi(0)$, and the last index $i_1$ is relabeled $\pi(k)$ under the constraint that $\pi(0)$ = $\pi(k)$.}
\begin{align}
\begin{split}
\pi:\{0,1,2,\dots,k\}\rightarrow \{1,2,\dots,N\} \text{ such that } \pi(0)\ = \ \pi(k).
\end{split}
\end{align}
We call an input variable index $L(\pi(i-1),\pi(i))$ an \textit{L-value}.  Each L-value must occur at least twice in a circuit to contribute to the sum, since the expected value of a product of independent random variables is the product of the expected values and the distribution $p$ is assumed to have mean 0.  An equivalence class of circuits is a partition of the set $\{1,2,\dots,k\}$, labeled by a \textit{word} of length $k$, where the first occurrence of each letter in the word is in alphabetical order.  If we let $w[i]$ for $1\le i\le k$ denote the $i\textsuperscript{th}$ entry of the word $w$, the equivalence class of circuits corresponding to $w$ is
\begin{align}
\begin{split}
\Pi(w) \ := \  \left\{\pi: w[i]=w[j] \Longleftrightarrow  L(\pi(i-1),\pi(i))=L(\pi(j-1),\pi(j)) \right\}.
\end{split}
\end{align}
The \textit{size} of $w$, or number of distinct letters, is denoted by $|w|$:
\begin{align}
\begin{split}
|w| \ := \  \#\{L(\pi(i-1),\pi(i)):1\le i \le k\}.
\end{split}
\end{align}
The positions of the letters in a word $i$, for $1\le i \le k$, along with the additional value $i = 0$, are called \textit{vertices}.\footnote{``Vertex'' and ``index'' are essentially synonymous in this paper.  Indices $\pi(i)$ that are functions of a generating vertex $i$ will simply be called ``generating indices'' and correspond to the degrees of freedom.}  A vertex is \textit{generating} if either $i = 0$ or $w[i]$ is the first occurrence of a letter in the word.  Otherwise, the vertex is \textit{non-generating}.  For example, if $w = ababcb$, then the generating vertices are $\{0,1,2,5\}$ and the non-generating vertices are $\{3,4,6\}$.
The number of generating vertices is equivalent to the maximum number of degrees of freedom one has in choosing a circuit that corresponds to that word, because once the indices $\pi(i)$ corresponding to generating vertices $i$ are chosen, the indices corresponding to the non-generating vertices are fixed by the fact that they have to satisfy matched L-values.  For example, consider the word $abab$ for the Toeplitz link function $L(i,j) = |i-j|$.  The word dictates the following system of equations:
\begin{align}
\begin{split}
|\pi(0)-\pi(1)|\ =\ |\pi(2)-\pi(3)| \text{ and } |\pi(1)-\pi(2)|\ =\ |\pi(3)-\pi(4)|.
\end{split}
\end{align}
We can choose the indices $\pi(0),\hspace{1mm} \pi(1), \text{ and } \pi(2)$ freely, since the vertices 0, 1, and 2 are generating, but then $\pi(3)$ is fixed by the matching constraints and $\pi(4)$ is defined to be equal to $\pi(0)$.  Since there are at most $N$ choices for each generating index corresponding to a generating vertex, and since there are $|w|+1$ generating vertices, the size of the equivalence class for word $w$ is at most
\begin{align}
\begin{split}
\#\Pi(w)\ = \ O(N^{|w|+1}).
\end{split}
\end{align}
Only \textit{pair-matched} words contribute, words in which every letter appears exactly twice, and odd moments are zero.  A sufficient condition for this is called Property B:\footnote{See \cite{B} for proof of this standard result.}
\begin{align}
\begin{split}
\Delta(L) \ := \  \text{sup}_N \text{sup}_{t \in \mathbb{Z}}\text{sup}_{1\le k \le N}\hspace{1mm}\#\{m: 1 \le m \le N, L(k,m)=t\}< \infty.
\end{split}
\end{align}
For a matrix satisfying Property B, its $\Delta(L)$ value, the maximum number of repetitions of the same random variable in any row or column, is finite.  Since only even moments with pair-matched words of length $2k$ contribute, the moments of the limiting spectral distribution can be written as
\begin{align}\label{1.21}
\begin{split}
M_{2k}\ = \ \text{lim}_{N\rightarrow \infty}M_{2k}(N)\ = \ \sum_{w:\hspace{1mm}w\text{ is pair-matched of length 2$k$}}\text{lim}_{N\rightarrow \infty}\frac{1}{N^{k+1}}\#\Pi(w).
\end{split}
\end{align}
Computing the limiting moments reduces to checking all possible pair-matched words, and for each word, finding the number circuits corresponding to that word.  Counting the number of circuits for a given word becomes equivalent to counting the number of integer solutions to a set of Diophantine, or integer-valued, equations.\footnote{In this paper we only consider matrix ensembles that satisfy Property B.}
Pair-matched words of length $2k$ can be classified as \textit{non-crossing partitions} or \textit{crossing partitions}.  Consider the set $\{1,2,\dots,2k\}$.  Arrange the elements on a circle sequentially.  Pick any pair partition of this set and draw an edge connecting every pair.  The partition is said to be non-crossing if none of the edges crosses another, and crossing otherwise.
The non-crossing partitions are in bijection with Catalan words (see for example \cite{AGZ}).  A pair-matched word of length $2k$ is called a \textit{Catalan word} if (1) there is at least one double letter, (2) if any double letter is deleted, the remaining word of length $2k-2$ is either empty or has a double letter, and (3) repeating the process in the previous step ultimately leads to an empty word.  For example, $aabbcc$ is a Catalan word, while $abcabc$ is not a Catalan word, as shown in Figure \ref{fig:figure1}.
\begin{figure}
\centering
\includegraphics[scale = 0.6]{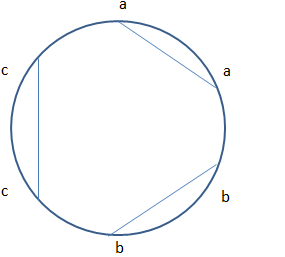}
\includegraphics[scale = 0.6]{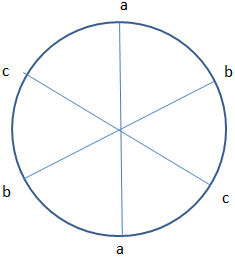}
\caption{\textit{aabbcc} on the left is a Catalan word, while \textit{abcabc} on the right is not a Catalan word. }\label{fig:figure1}
\end{figure}
The number of Catalan words of length $2k$ is given by the $2k\textsuperscript{th}$ \textit{Catalan number}
\begin{align}
\begin{split}
C_{2k} \ := \  \frac{1}{k+1}\binom{2k}{k}.
\end{split}
\end{align}
The $2k\textsuperscript{th}$  moment of the semicircle measure is exactly $C_{2k}$, with odd moments zero (see \cite{B}).  This is critical, because then one can prove that an ensemble of matrices has a semicircular limiting spectral distribution by showing that all Catalan words of length $2k$ contribute one to the $2k\textsuperscript{th}$ moment and all non-Catalan words contribute zero.
In this paper we investigate several types of convergence.
\begin{enumerate}
\item (Almost sure convergence) For each $k$, $M_{k}\left(A_N\right)\rightarrow M_{k}$ almost surely if
\begin{align}\label{eq:eq118}
\begin{split}
\text{Prob}\left(\{A_N\in \Omega_N : M_k\left(A_N\right)\rightarrow M_k\text{ as }N\rightarrow\infty\}\right)\ = \ 1.
\end{split}
\end{align}
\item (Convergence in probability) For each $k$, $M_k\left(A_N\right)\rightarrow M_k$ in probability if for all $\epsilon>0$,
\begin{align}\label{eq:eq119}
\begin{split}
\text{lim}_{N\rightarrow\infty}\text{Prob}\left(\{A_N\in\Omega_N:|M_{k}\left(A_N\right) - M_{k}|>\epsilon\}\right)\ = \ 0.
\end{split}
\end{align}
\item (Weak convergence) For each $k$, $M_{k}\left(A_N\right)\rightarrow M_{k}$ weakly if
\begin{align}
\begin{split}
\text{Prob}\left(M_k\left(A_N\right)\le x\right)\rightarrow\text{Prob}\left(M_k\le x\right)
\end{split}
\end{align}
as $N\rightarrow\infty$ for all $x$ at which $F_{M_k}(x)\ = \ \text{Prob}\left(M_k\le x\right)$ is continuous.
\end{enumerate}
The convergence is \textit{universal} if it is independent of $p$.
%%%%%%%%%%%%%%%%%%%%%%%%%%%%%%%%%%%%%%%%%%%%%%%%%%%%%%%%%%%%%%%%%%%%%%%%%%%%%%%%%%%%%%%%%%%%%%%%%%%%%%%%%%%%%%%%%%%%%%%%%%%%%%
%%%%%%%%%%%%%%%%%%%%%%%%%%%%%%%%%%%%%%%%%%%%%%%%%%%%%%%%%%%%%%%%%%%%%%%%%%%%%%%%%%%%%%%%%%%%%%%%%%%%%%%%%%%%%%%%%%%%%%%%%%%%%%
%%%%%%%%%%%%%%%%%%%%%%%%%%%%%%%%%%%%%%%%%%%%%%%%%%%%%%%%%%%%%%%%%%%%%%%%%%%%%%%%%%%%%%%%%%%%%%%%%%%%%%%%%%%%%%%%%%%%%%%%%%%%%%
%%%%%%%%%%%%%%%%%%%%%%%%%%%%%%%%%%%%%%%%%%%%%%%%%%%%%%%%%%%%%%%%%%%%%%%%%%%%%%%%%%%%%%%%%%%%%%%%%%%%%%%%%%%%%%%%%%%%%%%%%%%%%%
\subsection{Results}\label{sec:sec13}
Our main results concern the limiting spectral distributions of random matrices with a certain type of real-valued bivariate polynomial link function.  Using the Method of Moments, we prove the following.
\begin{thm}[Generalized Toeplitz and Hankel Matrices]\label{thm:thm14}
For fixed positive integers $\alpha$ and $\beta$, the generalized Toeplitz link function is defined as
\begin{align} L_{T,\alpha,\beta}(i,j) \ := \   \begin{cases}
       \alpha i-\beta j  & i \le j \\
       -\beta i + \alpha j & i > j,
     \end{cases}
\end{align}
while the generalized Hankel link function is defined as
\begin{align} L_{H,\alpha,\beta}(i,j) \ := \   \begin{cases}
       \alpha i+\beta j  & i \le j \\
       \beta i + \alpha j & i > j.
     \end{cases}
\end{align}
For both of these ensembles:
\begin{enumerate}
\item The fourth and sixth moments are functions of the parameters $\alpha$ and $\beta$.
\item For fixed $\alpha$ and $\beta$, as $N\rightarrow\infty$ the empirical spectral measures converge in probability and almost surely to a unique and universal limiting distribution.  The limiting distribution is non-semicircular, because for any $\alpha$ and $\beta$ the sixth moment does not agree with its respective Catalan number (see Lemma \ref{lem:lemma37}).
\item In the limit as $\alpha$ or $\beta$ tends to infinity, the limiting distributions converge to the semicircle distribution.\footnote{See \S\ref{sec:sec3} for the appropriate notion of convergence.} 
\end{enumerate}
\end{thm}
As in other related ensembles, it is very difficult to obtain closed-form expressions for the general moments of the limiting spectral distribution.  In the following theorem, however, we successfully identify the limiting distribution for other polynomial link functions.
\begin{thm}[Polynomial Toeplitz and Hankel Matrices]\label{thm:thm15}
Let $p_1(x) = a_mx^m+a_{m-1}x^{m-1}+\cdots+a_0$ and $p_2(x) = b_nx^n+b_{n-1}x^{n-1}+\cdots+b_0$ be polynomials with integer coefficients and $m\ne n$.  The link function for polynomial Toeplitz matrices is 
\begin{align}L_{PT}(i,j) \ := \  \begin{cases}
p_1(i) - p_2(j) & i \le j\\
-p_2(i) + p_1(j) & i > j,
\end{cases}
\end{align}
while the link function for polynomial Hankel matrices is
\begin{align}L_{PH}(i,j) \ := \  \begin{cases}
p_1(i) + p_2(j) & i \le j\\
p_2(i) + p_1(j) & i > j.
\end{cases}
\end{align}
For both of these ensembles, as $N\rightarrow\infty$ the empirical spectral measures converge in probability and almost surely to a universal semicircle distribution.\end{thm}
The rest of the paper is organized as follows.  We first establish some basic results about a general class of bivariate polynomial link functions in \S\ref{sec:sec2}.  We then analyze in detail the even moments of the generalized Toeplitz and Hankel matrices in \S\ref{sec:sec3} and prove the limiting distribution of polynomial Toeplitz and Hankel matrices in \S\ref{sec:sec4}.  We prove our convergence claims in \S\ref{sec:sec5}.  Future work is discussed in \S\ref{sec:sec6}, and numerical methods are discussed in Appendix \ref{sec:sec7}.  It is worth noting that numerical investigations played an important role in our analysis, as they highlighted key features of the combinatorics.
%%%%%%%%%%%%%%%%%%%%%%%%%%%%%%%%%%%%%%%%%%%%%%%%%%%%%%%%%%%%%%%%%%%%%%%%%%%%%%%%%%%%%%%%%%%%%%%%%%%%%%%%%%%%%%%%%%%%%%%%%%%%%%
%%%%%%%%%%%%%%%%%%%%%%%%%%%%%%%%%%%%%%%%%%%%%%%%%%%%%%%%%%%%%%%%%%%%%%%%%%%%%%%%%%%%%%%%%%%%%%%%%%%%%%%%%%%%%%%%%%%%%%%%%%%%%%
%%%%%%%%%%%%%%%%%%%%%%%%%%%%%%%%%%%%%%%%%%%%%%%%%%%%%%%%%%%%%%%%%%%%%%%%%%%%%%%%%%%%%%%%%%%%%%%%%%%%%%%%%%%%%%%%%%%%%%%%%%%%%%
%%%%%%%%%%%%%%%%%%%%%%%%%%%%%%%%%%%%%%%%%%%%%%%%%%%%%%%%%%%%%%%%%%%%%%%%%%%%%%%%%%%%%%%%%%%%%%%%%%%%%%%%%%%%%%%%%%%%%%%%%%%%%%
%%%%%%%%%%%%%%%%%%%%%%%%%%%%%%%%%%%%%%%%%%%%%%%%%%%%%%%%%%%%%%%%%%%%%%%%%%%%%%%%%%%%%%%%%%%%%%%%%%%%%%%%%%%%%%%%%%%%%%%%%%%%%%
%%%%%%%%%%%%%%%%%%%%%%%%%%%%%%%%%%%%%%%%%%%%%%%%%%%%%%%%%%%%%%%%%%%%%%%%%%%%%%%%%%%%%%%%%%%%%%%%%%%%%%%%%%%%%%%%%%%%%%%%%%%%%%
\section{Bivariate Polynomial Link Functions}\label{sec:sec2}
In this section we introduce a general class of real-valued bivariate polynomial link functions and prove several results that describe contributions to the moments.
Consider the link function
\begin{align}\label{2.1} L(i,j)\ = \ \begin{cases}
p_1(i) \pm p_2(j) & i \le j \\
\pm p_2(i) + p_1(j) & i > j,
\end{cases} \end{align}
where $p_1(x) = a_mx^m+a_{m-1}x^{m-1}+\cdots+a_0$ and $ p_2(x) = b_nx^n+b_{n-1}x^{n-1}+\cdots+b_0$ are polynomials\footnote{Assume the polynomials are nonconstant to avoid a violation of Property B.} with integer coefficients\footnote{We do not consider irrational coefficients.  Consider, for example, $p_1(x)=\sqrt{2}x$ and $p_2(x)=x$.  Because every resulting L-value in the upper triangle of the matrix is unique, matrices with this link function are equivalent to real symmetric matrices.  We also do not consider rational coefficients that are not integers.  Because matrix structure is invariant under multiplying the link function by a constant, just multiply the link function by the least common multiple of the denominators of the rational coefficients to get all integers.}.  It is easy to see that this generalizes the real symmetric Toeplitz matrices.  The link function splits any matrix into two zones, where Zone 1 is the upper triangle including the main diagonal and Zone 2 is the lower triangle.\footnote{We could have defined the zones to exclude the main diagonal, as the values of the main diagonal do not affect the limiting distribution of the eigenvalues.  Likewise, we could have set the main diagonal to be all zeros.}  For example, if two matched entries $a_{L(i_1,i_2)} = a_{L(i_3,i_4)}$ are such that $a_{i_1i_2},\hspace{1mm}a_{i_3i_4}\in\text{Zone 1}$, their L-values must satisfy
\begin{align}
\begin{split}
p_1\left(\pi(0)\right)\pm p_2\left(\pi(1)\right)\ = \ p_1\left(\pi(2)\right) \pm p_2\left(\pi(3)\right).
\end{split}
\end{align}
To compute moments, we count the number of contributing circuits for all pair-matched words.  For the letters in these words, we define an \textit{adjacent pair} as a matching between matrix entries that share one index, as in $a_{L(i_1,i_2)} = a_{L(i_2,i_3)}$. The following reduces the number of contributing circuits.
\begin{lem}[Adjacent Pairs]\label{lem:lemma21} Let L(i,j) be a bivariate polynomial link function where $p_1(x) = a_mx^m+a_{m-1}x^{m-1}+\cdots+a_0$ and $p_2(x) = b_nx^n+b_{n-1}x^{n-1}+\cdots+b_0$  are polynomials with integer coefficients.  Then in any contributing word of length 2k, adjacent pairs must be in opposite zones when $p_1(x)\ne p_2(x)$.\end{lem}
\begin{proof}
Any adjacent pair has the form $a_{L(i_1,i_2)} = a_{L(i_2,i_3)}$.  For example, assume that $a_{i_1i_2}$, $a_{i_2i_3}\in\text{ Zone 1}$ and the coefficient of $p_2(x)$ is positive.  The corresponding L-value equation has the form
\begin{equation}
p_1\left(\pi(0)\right) + p_2\left(\pi(1)\right)\ = \ p_1\left(\pi(1)\right) + p_2\left(\pi(2)\right).
\end{equation}
Choose $a_{i_1i_2}$ to correspond to the first letter in the word from which the adjacent pair is chosen.\footnote{For any word, we are free to cycle the letters in the word without changing the underlying configuration, since the corresponding pair partition is invariant under a rotation.}  Then $\pi(0)$ and $\pi(1)$ are generating indices.  Choose $k$ of the generating indices, leaving out $\pi(1)$.  This occurs with at most $k$ degrees of freedom.  Then the non-generating indices are fixed by the other L-value equations derived from the word, including $\pi(2)$, with a total of $\Delta\left(L\right)^k$ choices.  Now both $\pi(0)$ and $\pi(2)$ are chosen, and there remains a finite number of choices independent of $N$ for $\pi(1)$.  This constitutes a loss in degrees of freedom, leaving at most $O\left(N^k\right)$ solutions and a contribution of zero to the moment, since we divide by $N^{k+1}$.  A similar argument holds if both entries are in Zone 2, and it also applies if the coefficient of $p_2(x)$ is negative.
\end{proof}
Recall that words can be categorized as Catalan or non-Catalan.  The following describes the contribution of Catalan words to even moments.
\begin{lem}[Adjacent Lifting]\label{lem:lemma22} Let L(i,j) be a bivariate polynomial link function where $p_1(x) = a_mx^m+a_{m-1}x^{m-1}+\cdots+a_0$ and $p_2(x) = b_nx^n+b_{n-1}x^{n-1}+\cdots+b_0$ are polynomials with integer coefficients that are monotonic on $\mathbb{N}$.  Then any Catalan word of length $2k$ for matrices with such a link function contributes one to the $2k\textsuperscript{th}$ moment when $p_1(x)\ne p_2(x)$.\end{lem}
\begin{proof}
We begin by counting the number of circuits for the word \textit{aabb}, the only Catalan word of length 4 up to rotation.  Adjacent pairs must be in opposite zones, by Lemma \ref{lem:lemma21}, giving four sets of inequalities relating the generating indices $\pi(0)$, $\pi(1)$, and $\pi(3)$.  Following standard calculations outlined in \cite{B}, change variables to $v_x = \frac{\pi(x)}{N}$, where $v_x\in\{\frac{1}{N},\dots,\frac{N}{N}\}$.  In the limit of large $N$, integrating over the region specified by the inequalities shows that the structure \textit{aabb} contributes one to the fourth moment:
\begin{align}
\begin{split}
 \int_0^1\int_0^1\int_0^1dv_0dv_1dv_3\ = \ 1.
\end{split}
\end{align}
For higher moments, any non-crossing pair partition must have at least one adjacent pair of the form $a_{L(i_1,i_2)}=a_{L(i_2,i_3)}$, by the definition of a Catalan word.  Since adjacent pairs must be located in opposite zones, any such adjacent pair must require $\pi(0) = \pi(2)$.  Since there are two sets of zones for the pair, the remaining index is bound either by $\pi(1)\le \pi(0)$ or $\pi(1)>\pi(0)$, leaving $\pi(1)$ as a free index.  ``Lift'' this pair by setting $\pi(0) = \pi(2)$ and relabeling the remaining indices appropriately.  Since there are now $2k-2$ indices left, the remaining structure is a Catalan word for the $(2k-2)\textsuperscript{th}$ moment, and the contribution can be computed with $\pi(1)$ as an extra degree of freedom.   Using this process, any non-crossing pair partition can be reduced to the fourth moment structure, and since that structure contributes one, any other Catalan word of length $2k$ contributes one to the $2k\textsuperscript{th}$ moment. \end{proof}
\begin{lem}[Zeroth, Second, and Odd Moments]\label{ref:lemma23} Let L(i,j) be a bivariate polynomial link function where $p_1(x) = a_mx^m+a_{m-1}x^{m-1}+\cdots+a_0$ and $p_2(x) = b_nx^n+b_{n-1}x^{n-1}+\cdots+b_0$  are polynomials with integer coefficients.  Then the zeroth moment is 1, the second moment is 1, and odd moments are zero.\end{lem}
\begin{proof}
Standard calculations in \cite{HM} apply for the zeroth and second moments.  Odd moments are zero because the link function satisfies Property B.  For any fixed $i$ or $j$, $L(i,j)$ is a polynomial in either $j$ or $i$ and so takes on any given value only finitely many times.\footnote{In fact, this is true for any bivariate polynomial $p(x,y)=a_nx^ny^n+a_{n-1}x^{n-1}y^{n-1}+\cdots+a_0.$}
\end{proof}
%%%%%%%%%%%%%%%%%%%%%%%%%%%%%%%%%%%%%%%%%%%%%%%%%%%%%%%%%%%%%%%%%%%%%%%%%%%%%%%%%%%%%%%%%%%%%%%%%%%%%%%%%%%%%%%%%%%%%%%%%%%%%%
%%%%%%%%%%%%%%%%%%%%%%%%%%%%%%%%%%%%%%%%%%%%%%%%%%%%%%%%%%%%%%%%%%%%%%%%%%%%%%%%%%%%%%%%%%%%%%%%%%%%%%%%%%%%%%%%%%%%%%%%%%%%%%
%%%%%%%%%%%%%%%%%%%%%%%%%%%%%%%%%%%%%%%%%%%%%%%%%%%%%%%%%%%%%%%%%%%%%%%%%%%%%%%%%%%%%%%%%%%%%%%%%%%%%%%%%%%%%%%%%%%%%%%%%%%%%%
%%%%%%%%%%%%%%%%%%%%%%%%%%%%%%%%%%%%%%%%%%%%%%%%%%%%%%%%%%%%%%%%%%%%%%%%%%%%%%%%%%%%%%%%%%%%%%%%%%%%%%%%%%%%%%%%%%%%%%%%%%%%%%
%%%%%%%%%%%%%%%%%%%%%%%%%%%%%%%%%%%%%%%%%%%%%%%%%%%%%%%%%%%%%%%%%%%%%%%%%%%%%%%%%%%%%%%%%%%%%%%%%%%%%%%%%%%%%%%%%%%%%%%%%%%%%%
%%%%%%%%%%%%%%%%%%%%%%%%%%%%%%%%%%%%%%%%%%%%%%%%%%%%%%%%%%%%%%%%%%%%%%%%%%%%%%%%%%%%%%%%%%%%%%%%%%%%%%%%%%%%%%%%%%%%%%%%%%%%%%
\section{Generalized Toeplitz and Hankel Matrices}\label{sec:sec3}  In this section we work towards proving Theorem \ref{thm:thm14}, in which we generalize the Toeplitz and Hankel link functions by changing the slope of the lines along which matrix entries are held constant.  In \S\ref{sec:sec31} we examine the generalized Toeplitz matrices and in \S\ref{sec:sec32} we analyze the generalized Hankel matrices.   
\subsection{Generalized Toeplitz Matrices}\label{sec:sec31} Recall that for fixed positive integers $\alpha$ and $\beta$, the generalized Toeplitz link function is\footnote{Note that this link function is of the form Eq. \eqref{2.1}.}
\begin{align} L_{T,\alpha,\beta}(i,j) \ := \   \begin{cases}
       \alpha i-\beta j  & i \le j \\
       -\beta i + \alpha j & i > j.
     \end{cases}
\end{align}
A matrix with $\alpha = \beta$ reduces to the original Toeplitz, while a $5\times5$ matrix with $\alpha = 2$ and $\beta = 1$ would have the structure
\begin{align} A_5\ = \ \left( \begin{array}{ccccc}
a_{1} & a_{0}& a_{-1}& a_{-2} & a_{-3}\\
a_{0} & a_{2}& a_{1}& a_{0} & a_{-1}\\
a_{-1} & a_{1}& a_{3}& a_{2} & a_{1}\\
a_{-2}& a_{0} & a_{2} & a_{4} & a_{3} \\
a_{-3} & a_{-1}& a_{1}& a_{3} & a_{5} \end{array} \right). \end{align}
\begin{figure}
\centering
\includegraphics[scale = 0.3]{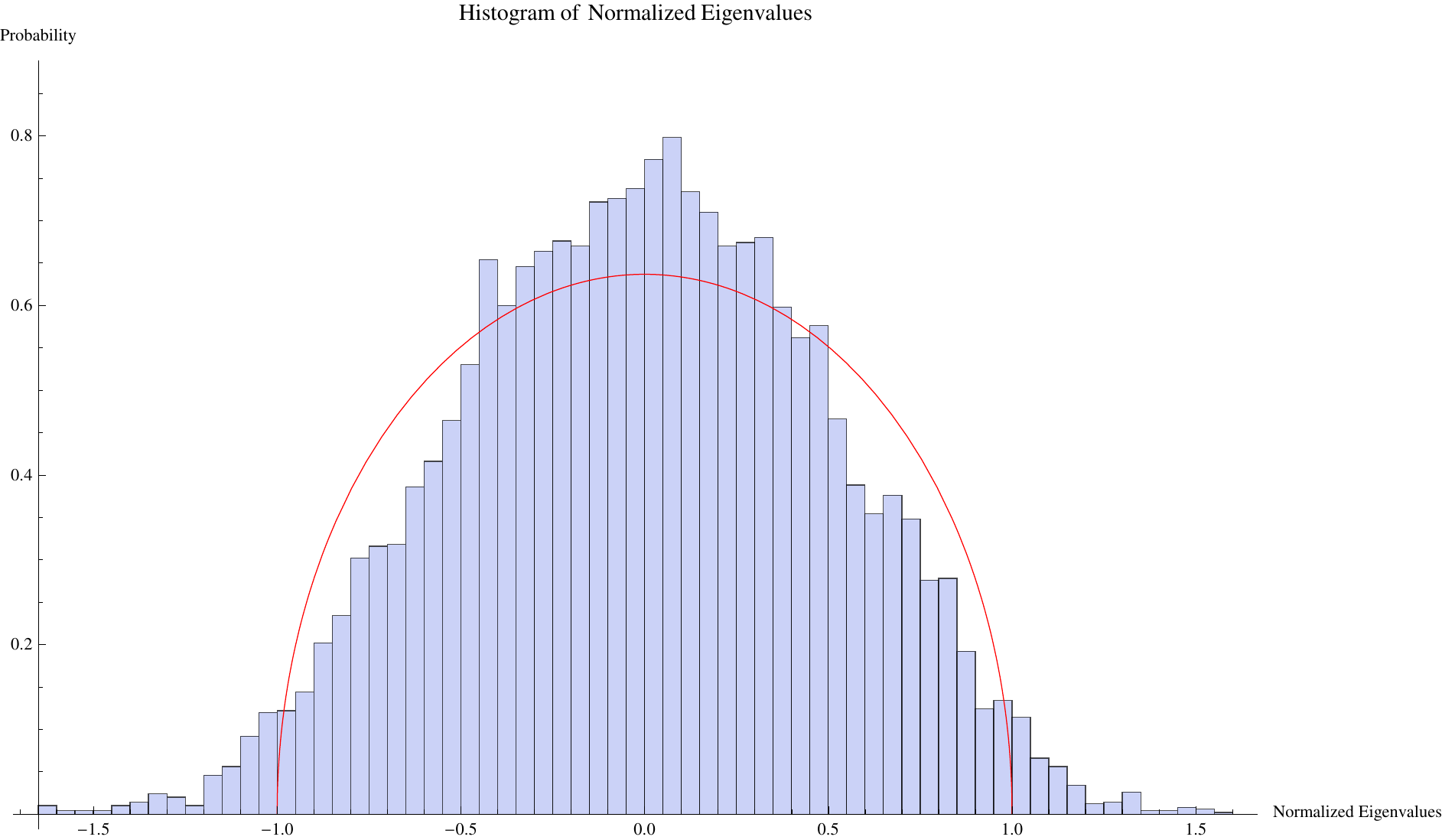}
\includegraphics[scale = 0.25]{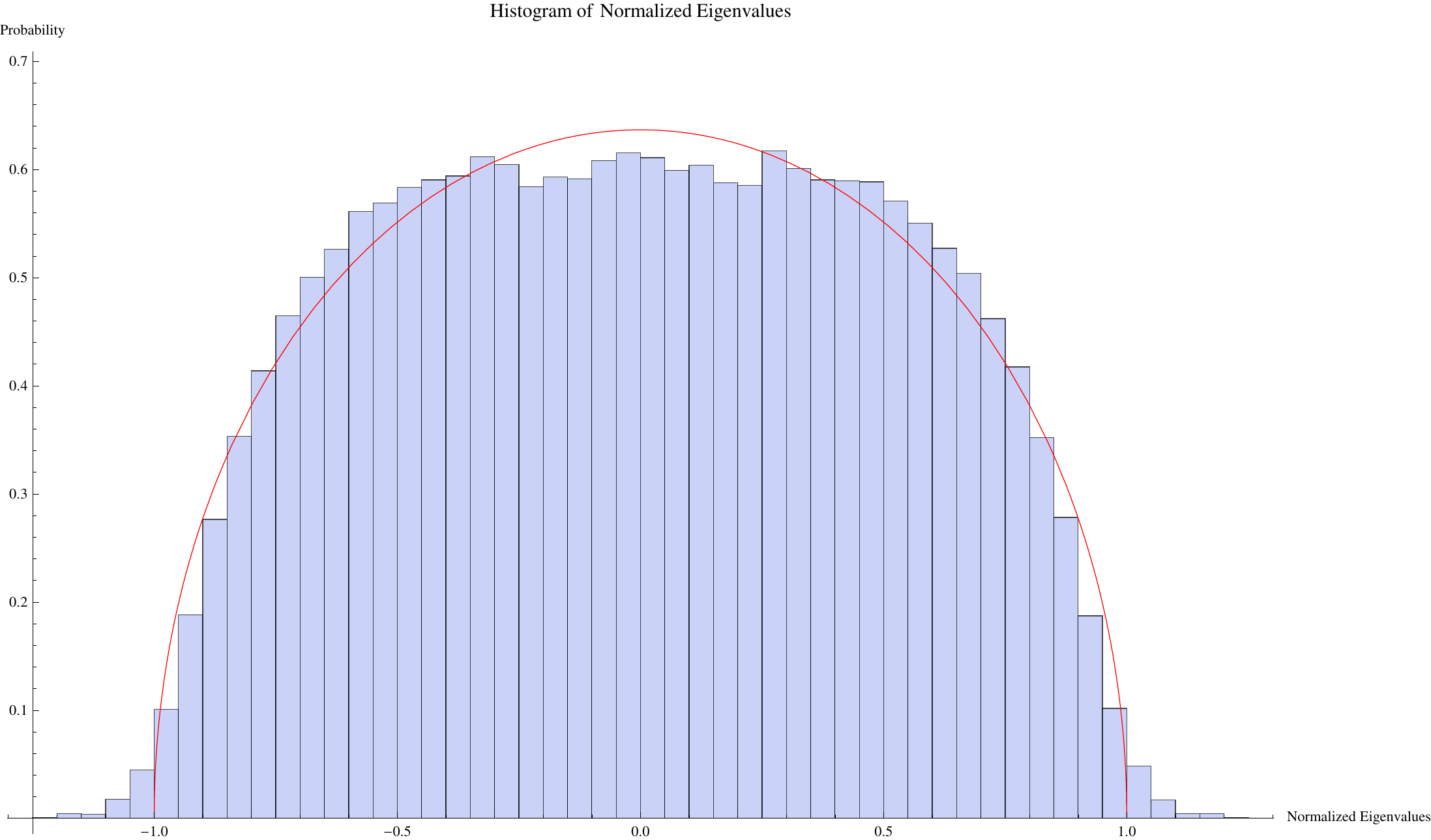}
\includegraphics[scale = 0.3]{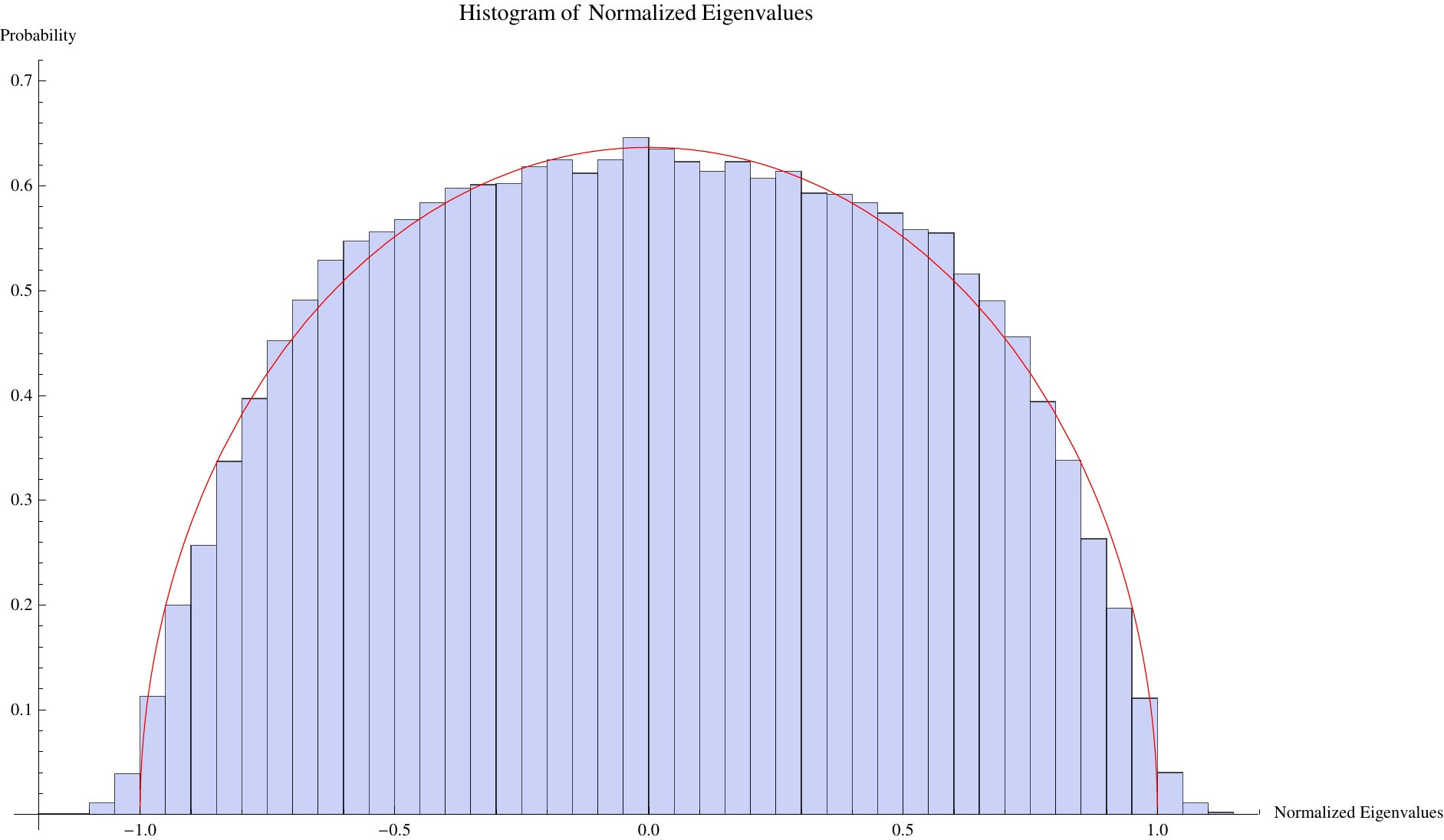}
\includegraphics[scale = 0.3]{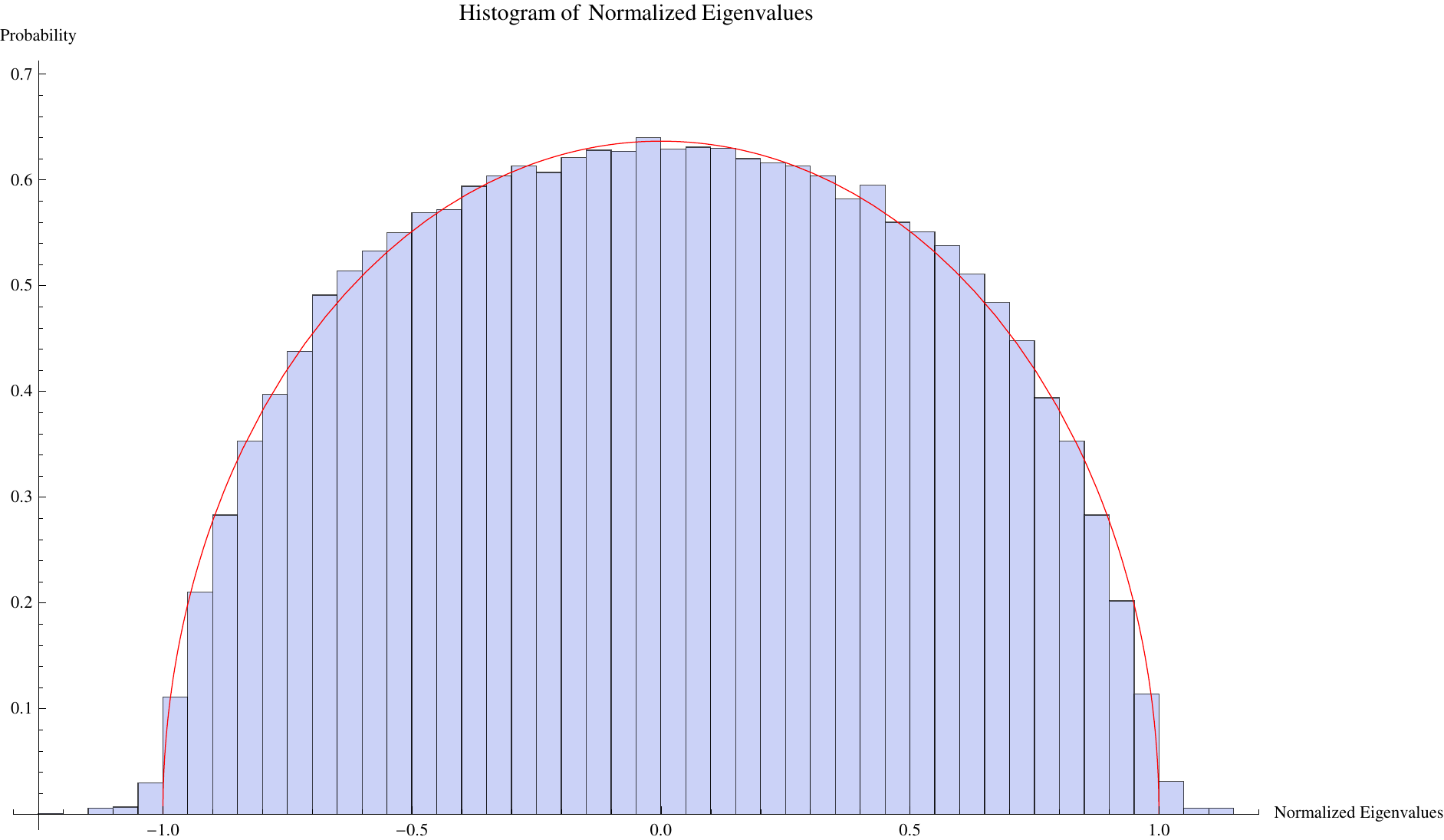}
\caption{Histograms of numerical eigenvalues of 100 generalized Toeplitz matrices of size $1200\times1200$.  Each has $\alpha = 1$.  From the upper left, going from left to right, $\beta$ is equal to 1, 2, 3, and 4.  The red curve is the semicircle distribution for an eigenvalue normalization of $2\sqrt{N}$ (see foonote \ref{foot:foot3}).  With increasing values of $\beta$, the distribution becomes more semicircular.}\label{fig:figure2}
\end{figure}
%%%%%%%%%%%%%%%%%%%%%%%%%%%%%%%%%%%%%%%%%%%%%%%%%%%%%%%%%%%%%%%%%%%%%%%%%%%%%%%%%%%%%%%%%%%%%%%%%%%%%%%%%%%%%%%%%%%%%%%%%%%%%%
%%%%%%%%%%%%%%%%%%%%%%%%%%%%%%%%%%%%%%%%%%%%%%%%%%%%%%%%%%%%%%%%%%%%%%%%%%%%%%%%%%%%%%%%%%%%%%%%%%%%%%%%%%%%%%%%%%%%%%%%%%%%%%
\subsubsection{Fourth Moment}
The pair-matched words of length four are \textit{aabb}, \textit{abba}, and \textit{abab}.  By Lemma \ref{lem:lemma22}, \textit{aabb} and \textit{abba} each contribute one to the fourth moment.\footnote{This holds for $\alpha = \beta$ and $\alpha \ne\beta.$  See \cite{HM} for the case $\alpha = \beta$, which corresponds to original Toeplitz matrices.}  The word \textit{abab} only contributes to the moment when $\alpha = \beta$.  For example, let $a_{i_1i_2}\in\text{ Zone 1}$,\hspace{1mm}$a_{i_2i_3}\in\text{ Zone 2}$,\hspace{1mm}$a_{i_3i_4}\in\text{ Zone 2}$,\hspace{1mm}\text{and }$a_{i_4i_1}\in\text{Zone 1}$ to get the system of equations
\begin{align}
\begin{split}
\alpha \pi(0)-\beta \pi(1) &\ = \ -\beta \pi(2) + \alpha \pi(3)\\
-\beta \pi(1) + \alpha \pi(2) &\ = \ \alpha \pi(3) - \beta \pi(0).
\end{split}
\end{align}
Assume $\alpha\ne\beta$.  Subtracting the two equations, we have
\begin{align}
\begin{split}
\alpha\left(\pi(0)-\pi(2)\right)=\beta\left(\pi(0)-\pi(2)\right),
\end{split}
\end{align}
for which there are no valid solutions.  Likewise, any other choice of zones either produces a similar obstruction or introduces an extra linear constraint on the generating indices that constitutes a loss in degrees of freedom.  Therefore, the word \textit{abab} does not contribute when $\alpha \ne \beta$.
For $\alpha = \beta$, we are reduced to the original Toeplitz matrices, and \cite{B} and \cite{HM} show that the contribution is $2/3$.  Let $M_k(T,\alpha,\beta)$ denote the $k\textsuperscript{th}$ moment of the generalized Toeplitz limiting spectral distribution. Thus, we have
\begin{align}\label{eq:eqref320}
M_4(T,\alpha,\beta)\ = \  \begin{cases}
       2 & \alpha\ne\beta \\\\
     2\frac{2}{3} & \alpha\ = \ \beta\hspace{2mm}\text{  (original Toeplitz)}.
     \end{cases}
\end{align}
%%%%%%%%%%%%%%%%%%%%%%%%%%%%%%%%%%%%%%%%%%%%%%%%%%%%%%%%%%%%%%%%%%%%%%%%%%%%%%%%%%%%%%%%%%%%%%%%%%%%%%%%%%%%%%%%%%%%%%%%%%%%%%
%%%%%%%%%%%%%%%%%%%%%%%%%%%%%%%%%%%%%%%%%%%%%%%%%%%%%%%%%%%%%%%%%%%%%%%%%%%%%%%%%%%%%%%%%%%%%%%%%%%%%%%%%%%%%%%%%%%%%%%%%%%%%%
\subsubsection{Sixth Moment}
The pair-matched words of length six are \textit{aabbcc}, \textit{aabccb}, \textit{aabcbc}, \textit{abacbc}, and \textit{abcabc}, along with other words that are isomorphic (equal up to a rotation) to these.  Respectively, there are 2, 3, 6, 3, and 1 versions for the configurations corresponding to these words.\footnote{For example, the Catalan configuration, or partition structure, shown in Figure \ref{fig:figure1} has two versions, the isomorphic words \textit{aabbcc}\text{ and }\textit{abbcca}.}
By Lemma \ref{lem:lemma22} the words \textit{aabbcc}, \textit{aabccb}, and other words isomorphic to either of them contribute one to the sixth moment.
For the word \textit{aabcbc}, we can ``lift'' the adjacent pair and relabel the remaining indices so that the fourth moment structure for the word \textit{abab} remains. Therefore, versions of the configuration corresponding to this word contribute 0 when $\alpha \ne \beta$ and $2/3$ when $\alpha = \beta$.
For the word \textit{abacbc}, if $\alpha \ne \beta$, there are no choices of zones that do not produce an obstruction or introduce an extra linear constraint on the generating indices.  When $\alpha = \beta$, \cite{B} and \cite{HM} compute the contribution to be $1/2$.
For the word \textit{abcabc}, there are two choices of zones that do not produce extra constraints.  The L-value equations for these cases are:
\begin{align}\label{3.11}
\begin{split}
-\beta \pi(0) + \alpha \pi(1) &= \alpha \pi(3) -\beta \pi(4)\\
\alpha \pi(1) - \beta \pi(2) &= -\beta \pi(4) + \alpha \pi(5)\\
-\beta \pi(2) + \alpha \pi(3)  &= \alpha \pi(5) - \beta \pi(0)
\end{split}
\end{align}
and
\begin{align}\label{3.12}
\begin{split}
\alpha \pi(0) - \beta \pi(1) &= -\beta \pi(3) + \alpha \pi(4)\\
-\beta \pi(1) + \alpha \pi(2) &= \alpha \pi(4) - \beta \pi(5)\\
\alpha \pi(2) - \beta \pi(3)  &= -\beta \pi(5) + \alpha \pi(0).
\end{split}
\end{align}
For Eq. \eqref{3.11} we can choose $\pi(0),\hspace{1mm}\pi(1),\hspace{1mm}\pi(2)\text{, and }\pi(3)$ freely.  Then $\pi(4)$ and $\pi(5)$ are fixed:
\begin{align}
\begin{split}
\pi(4) &= \pi(0)-\frac{\alpha}{\beta}\pi(1)+\frac{\alpha}{\beta}\pi(3)\\
\pi(5) &= \pi(3) -\frac{\beta}{\alpha}\pi(2)+\frac{\beta}{\alpha}\pi(0).
\end{split}
\end{align}
Using the same variable transformation as in Lemma \ref{lem:lemma22}, we count the contribution by integrating the indicator function, $\mathbb{I}$, acting on the regions defined above along with the choice of zones.  Letting $a = \frac{\beta}{\alpha}$, the contribution is
\begin{align}
\begin{split}
{} & \int_0^1\int_0^1\int_0^1\int_0^1\mathbb{I}(0 \le v_0-\frac{v_1}{a}+\frac{v_3}{a}\le1\text{ and }0 \le av_0-av_2+v_3\le1\text{ and }v_0>v_1\text{ and } \\
{} & \hspace{27mm}v_1<v_2\text{ and } v_2>v_3\text{ and } v_3<v_0-\frac{v_1}{a}+\frac{v_3}{a} \text{ and } v_0-\frac{v_1}{a}+\frac{v_3}{a}>av_0\\
{} & \hspace{27mm} -av_2+v_3 \text{ and } av_0-av_2+v_3\le v_0)dv_1dv_2dv_3dv_0,
\end{split}
\end{align}
which simplifies to
 \begin{equation}\label{3.16}\begin{cases}
       \frac{\alpha}{4}\frac{1}{\alpha+\beta} & \alpha<\beta \\\\
      \frac{\beta}{4}\frac{1}{\alpha+\beta} & \alpha > \beta.
     \end{cases}
\end{equation}
Since Eq. \eqref{3.12} is related to Eq. \eqref{3.11} by a permutation of the indices, the contribution from that case is also given by Eq. \eqref{3.16}.\footnote{$\pi(0)$, $\pi(1)$, and $\pi(2)$ change to $\pi(3)$, $\pi(4)$, and $\pi(5)$, respectively.}
Since there are 2 versions of the configuration corresponding to the word \textit{aabbcc} and 3 versions of the configuration corresponding to the word \textit{aabccb}, the Catalan words contribute a total of 5 to the moment.  Including the extra factors from the word \textit{abcabc}, the sixth moment is
 \begin{equation}M_6(T,\alpha,\beta)\ = \ \begin{cases}\label{eq:eqref321}
       5 + \frac{\alpha}{2}\frac{1}{\alpha+\beta} & \alpha<\beta \\\\
     5 + \frac{\beta}{2}\frac{1}{\alpha+\beta} & \alpha > \beta \\\\
     11 & \alpha\ = \ \beta\hspace{2mm}\text{  (original Toeplitz)}.
     \end{cases}
\end{equation}
%%%%%%%%%%%%%%%%%%%%%%%%%%%%%%%%%%%%%%%%%%%%%%%%%%%%%%%%%%%%%%%%%%%%%%%%%%%%%%%%%%%%%%%%%%%%%%%%%%%%%%%%%%%%%%%%%%%%%%%%%%%%%%
%%%%%%%%%%%%%%%%%%%%%%%%%%%%%%%%%%%%%%%%%%%%%%%%%%%%%%%%%%%%%%%%%%%%%%%%%%%%%%%%%%%%%%%%%%%%%%%%%%%%%%%%%%%%%%%%%%%%%%%%%%%%%%
\subsubsection{Existence of Higher Moments, Bounds, and Limiting Behavior}
Unfortunately, all higher moments for this link function become increasingly computationally intensive.  Although we cannot find a closed-form expression for higher moments, we can show that all higher moments exist and are finite.
\begin{lem}[Existence of Higher Moments]\label{lem:lemma31} If the probability distribution $p$ has mean 0, variance 1, and finite higher moments, then for all nonnegative integers $k$, $M_k = \text{lim}_{N\rightarrow\infty}M_k(N)$ exists and is finite for generalized Toeplitz matrices.\end{lem}
\begin{proof}
As described in \cite{X}, for any word $w$ of length $2k$ we obtain a system of linear equations relating the transformed variables $v_0,\dots,v_{2k-1}\in\{\frac{1}{N},\dots,\frac{N}{N}\}$, which determine a nice region in the $(k+1)$-dimensional unit cube.  As $N\rightarrow\infty$, we obtain the finite volume of this region, transform back to the variables $\pi(x)$, and then extract the finite limiting moment by dividing by $N^{k+1}$.
\end{proof}
\begin{lem}[Bounds on the Moments]\label{lem:lemma32}  Let $C_{2k}$ be the $2k\textsuperscript{th}$ moment of the semicircle distribution, $M_{2k}(T)$ the $2k\textsuperscript{th}$ moment of the Toeplitz ensemble limiting distribution, and $M_{2k}\left(T,\alpha,\beta\right)$ the $2k\textsuperscript{th}$ moment of the generalized Toeplitz ensemble limiting distribution.  Then $C_{2k}\le M_{2k}\left(T,\alpha,\beta\right) \le M_{2k}(T)$ for all nonnegative integers $k$.\end{lem}
\begin{proof}
Since each Catalan word contributes one, the moments are at least as large as the semicircle moments.  For a non-Catalan word, the $\alpha$ and $\beta$ in the link function decrease the contributions that otherwise occur when $\alpha = \beta$.  Let $g=\text{gcd}(\alpha,\beta)$.  Then for a given matrix entry, there are at most $\left\lceil\frac{N}{\max\left(\alpha,\beta\right)/g}\right\rceil$ L-matches in the upper triangle of the matrix.  When $\alpha \ne \beta$ there are fewer matchings than the Toeplitz case, for which $\alpha = \beta$, which means there cannot be more solutions to the relevant Diophantine equations, and hence the upper bound holds.
\end{proof}  
We can also show that in the limit as either $\alpha$ or $\beta$ becomes very large, and the other is fixed, the moments of the limiting distribution for the generalized Toeplitz ensemble approach those of the semicircle measure.\footnote{This is evidenced by Eq. \eqref{eq:eqref321}.}
\begin{lem}[Limiting Behavior]\label{lem:lemma33}
For fixed $\alpha,\hspace{1mm}\lim_{\beta \rightarrow \infty}  \hspace{2mm}M_{2k}(T,\alpha,\beta) = C_{2k}$, and for fixed $\beta,\hspace{1mm}\text{lim}_{\alpha \rightarrow \infty}  \hspace{2mm}M_{2k}(T,\alpha,\beta) = C_{2k}$  for generalized Toeplitz matrices when the limits are taken appropriately.\end{lem}
\begin{proof}
Consider the limit as $\beta \rightarrow \infty$.  To prove that the limiting spectral measure is a semicircle, it suffices to show that all non-Catalan words contribute zero.  Specifically, we will show that the contribution of a word that is fully crossed, or a word that has no adjacent pairs, is zero.  This shows that all non-Catalan words contribute zero, since the contribution of any non-Catalan word with adjacent pairs is calculated by ``lifting'' the adjacent pairs, and any loss in degrees of freedom in lower moments will propagate through these adjacent pairs.  We will let $\beta$ grow to infinity as a function of $N$, such that $\text{lim}_{N \rightarrow \infty}f(N) = \infty$. \par
Any fully crossed word $w$ has non-adjacent pairs of the form $a_{L(i_s,i_{s+1})} = a_{L(i_t,i_{t+1})}$ for some $s,t\in\mathbb{Z}^+$.  For any such pair there are four possible sets of zones for the L-value equation, each of which  results in one index with coefficient $\alpha$ on both sides of the equation and one index with coefficient $\beta$ on both sides of the equation.  Without loss of generality, then, we assume that $a_{i_si_{s+1}},\hspace{1mm}a_{i_ti_{t+1}}\in\text{ Zone 1}$.  In the following argument, we also want to choose entries to be specific letters in the fully crossed word.  We pick this matched pair such that all letters between the letters that correspond to these entries are distinct.  This must be possible, otherwise the word would have an adjacent pair.  For the resulting pair, choose the first index to be the generating index $\pi(0)$.  Then the next index, $\pi(1)$, is also a generating index, as it now corresponds to the first letter in the word.  Although we cannot know in general the location of the other indices, we just name them $\pi(2)$ and $\pi(3)$.  Since intermediate letters are distinct and the word began with generating index $\pi(0)$, $\pi(2)$ must correspond to the first occurrence of a letter in the word, making $\pi(2)$ another generating index.  Therefore, there are at least three generating indices in this pair.  If there are $x$ degrees of freedom in choosing all four indices, there are at most $x+k-2$ degrees of freedom for the word.  If we can show that $x<3$, then any non-Catalan word of length $2k$ contributes zero to the $2k\textsuperscript{th}$ moment, since there must be $k+1$ degrees of freedom for a contribution.\par
There are at most $N$ choices for $\pi(0)$ and $N$ choices for $\pi(1)$.  Then the number of choices for $\pi(2)$ and $\pi(3)$ is equivalent to the number of matchings in the matrix for entry $a_{i_1i_2}$, given fixed values for $\pi(0)$ and $\pi(1)$.  When $f(N)>\alpha$, there are at most $2\left\lceil\frac{N}{f(N)/g}\right\rceil$ matchings, where $g = \text{gcd}(\alpha,f(N))$.  To see if there are fewer than three degrees of freedom, we check if $\text{lim}_{N\rightarrow\infty}\frac{\#\Pi(w)}{N^3} = 0.$  Dropping the ceiling notation, we have
\begin{equation}
\text{lim}_{N\rightarrow\infty}\frac{\#\Pi(w)}{N^3}\propto \text{lim}_{N\rightarrow\infty}\frac{2N^3}{N^3f(N)/g}\ = \ \text{lim}_{N\rightarrow\infty}\frac{2g}{f(N)}\ = \ 0.
\end{equation}
This holds for every set of zones, so $x<3$.  A similar proof holds for $\alpha \rightarrow \infty$.
\end{proof}
%%%%%%%%%%%%%%%%%%%%%%%%%%%%%%%%%%%%%%%%%%%%%%%%%%%%%%%%%%%%%%%%%%%%%%%%%%%%%%%%%%%%%%%%%%%%%%%%%%%%%%%%%%%%%%%%%%%%%%%%%%%%%%
%%%%%%%%%%%%%%%%%%%%%%%%%%%%%%%%%%%%%%%%%%%%%%%%%%%%%%%%%%%%%%%%%%%%%%%%%%%%%%%%%%%%%%%%%%%%%%%%%%%%%%%%%%%%%%%%%%%%%%%%%%%%%%
\subsection{Generalized Hankel Matrices}\label{sec:sec32}
Recall that for fixed positive integers $\alpha$ and $\beta$, the generalized Hankel link function is\footnote{Note that this link function is of the form Eq. \eqref{2.1}.}
\begin{align}L_{H,\alpha,\beta}(i,j) \ := \  \begin{cases}
\alpha i + \beta j & i\le j\\
\beta i + \alpha j & i > j.
\end{cases}
\end{align}
A matrix with $\alpha = \beta$ reduces to the original Hankel matrix, while a $5\times5$ matrix with $\alpha = 2$ and $\beta = 1$ would have the structure
\begin{align} A_5\ = \ \left( \begin{array}{ccccc}
a_{3} & a_{4}& a_{5}& a_6 & a_{7}\\
a_{4} & a_{6}& a_{7}& a_8 & a_{9}\\
a_{5} &a_{7}& a_{9}& a_{10} & a_{11}\\
a_{6} & a_{8} & a_{10}& a_{12}& a_{13} \\
a_{7} & a_{9}& a_{11}& a_{13} & a_{15} \end{array} \right). \end{align}
\begin{figure}
\centering
\includegraphics[scale = 0.57]{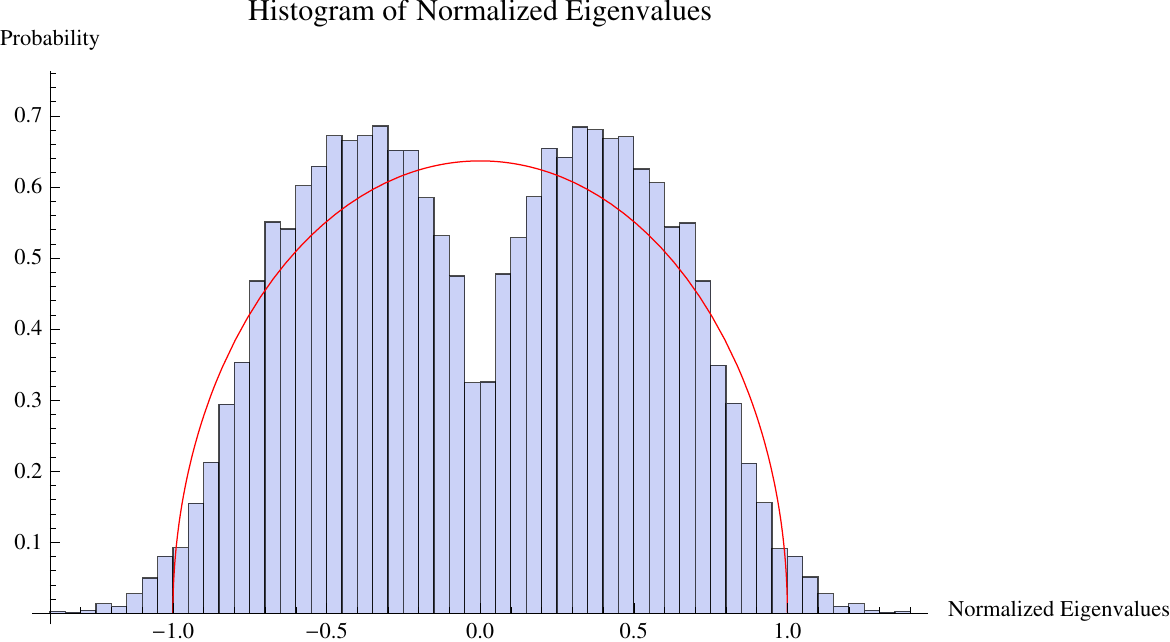}
\includegraphics[scale = 0.35]{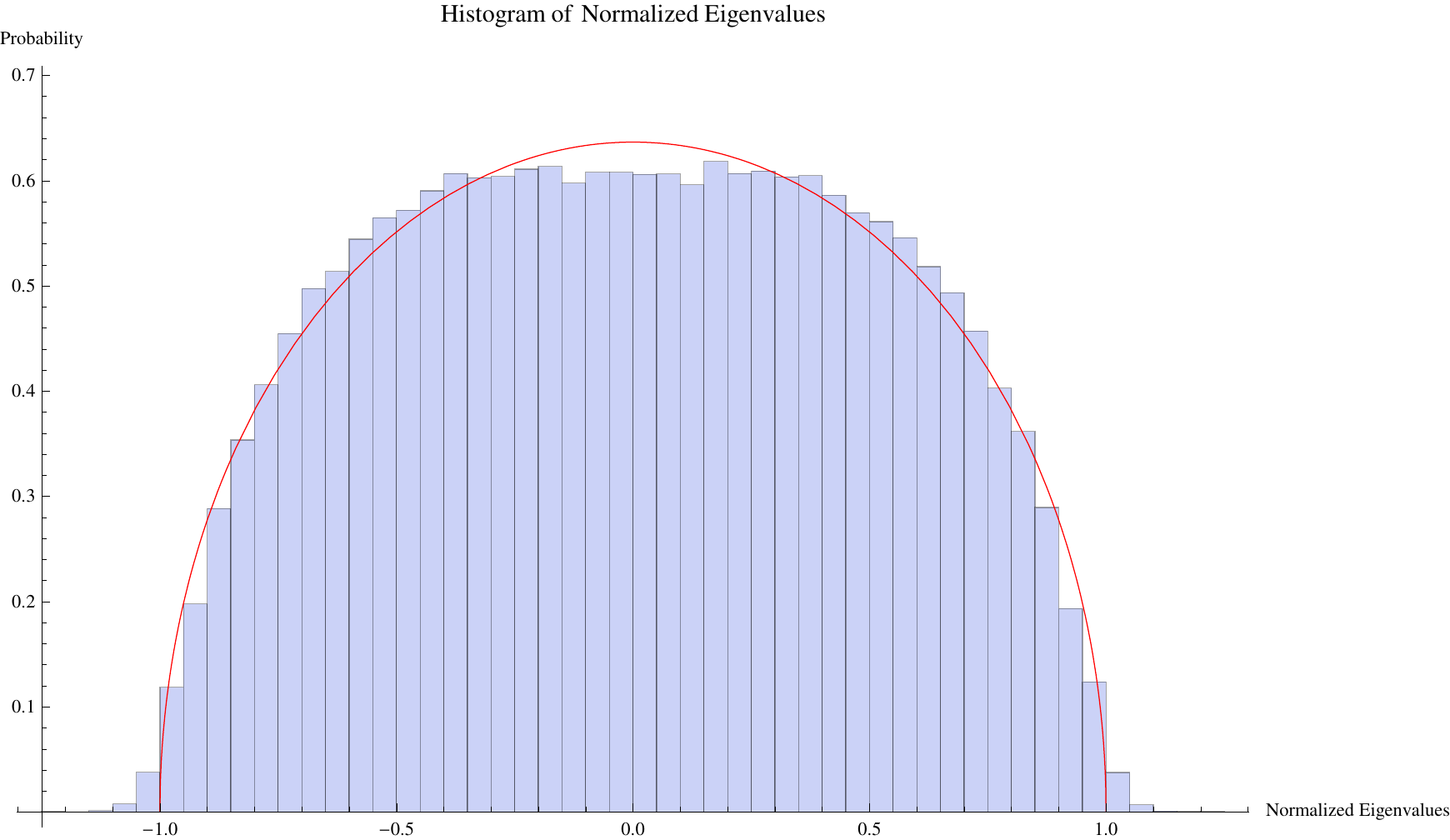}
\includegraphics[scale = 0.35]{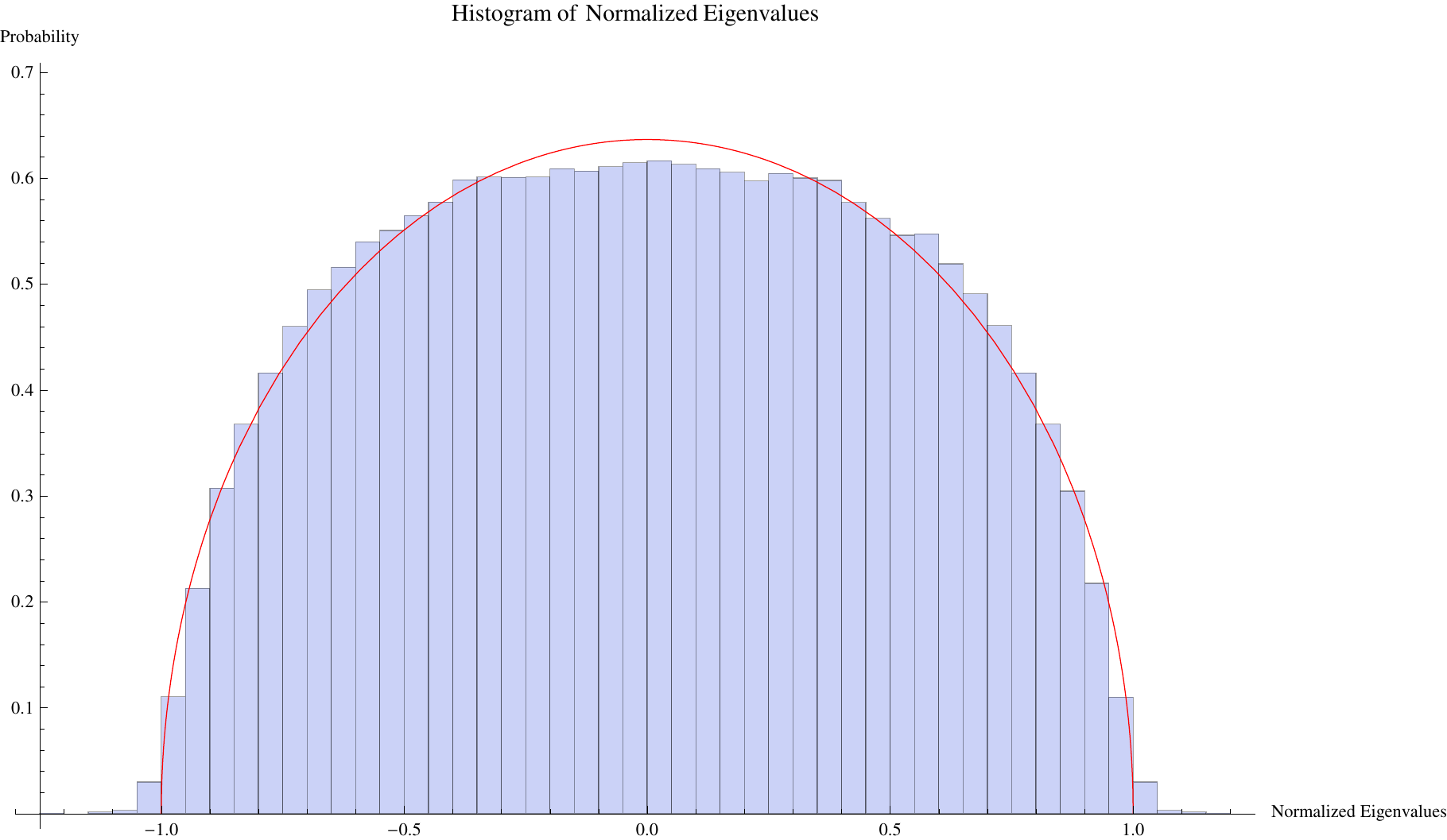}
\includegraphics[scale = 0.30]{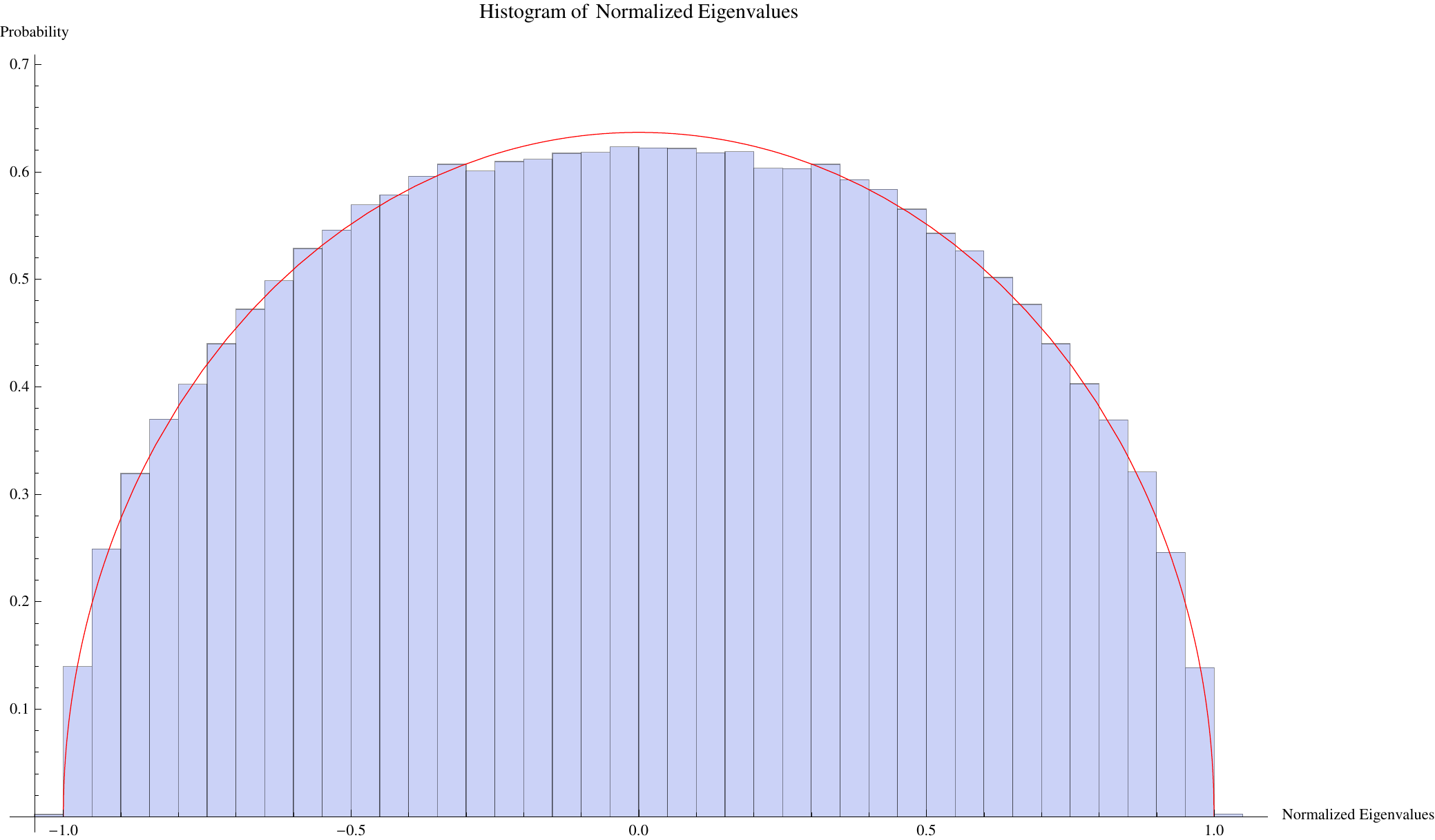}
\caption{Histograms of numerical eigenvalues of 100 generalized Hankel matrices of size $1200\times1200$.  Each has $\alpha = 1$.  From the upper left, going from left to right, $\beta$ is equal to 1, 2, 3, and 4.  The red curve is the semicircle distribution for an eigenvalue normalization of $2\sqrt{N}$ (see footnote \ref{foot:foot3}).  With increasing values of $\beta$, the distribution becomes more semicircular.}
\end{figure}
%%%%%%%%%%%%%%%%%%%%%%%%%%%%%%%%%%%%%%%%%%%%%%%%%%%%%%%%%%%%%%%%%%%%%%%%%%%%%%%%%%%%%%%%%%%%%%%%%%%%%%%%%%%%%%%%%%%%%%%%%%%%%%
%%%%%%%%%%%%%%%%%%%%%%%%%%%%%%%%%%%%%%%%%%%%%%%%%%%%%%%%%%%%%%%%%%%%%%%%%%%%%%%%%%%%%%%%%%%%%%%%%%%%%%%%%%%%%%%%%%%%%%%%%%%%%%
\subsubsection{Fourth Moment}
By Lemma \ref{lem:lemma22}, all Catalan words contribute one.\footnote{This holds for $\alpha = \beta$ and $\alpha \ne \beta$.  See \cite{B} for the case $\alpha = \beta$, which corresponds to original Hankel matrices.}  Now consider the word \textit{abab}.  By the same argument that applied in the generalized Toeplitz case, there will be an extra constraint in the L-value equations if $\alpha \ne \beta$.  For $\alpha = \beta$, we are reduced to original Hankel matrices, and \cite{B} shows that the contribution for this word is zero.  Thus, if $M_k(H,\alpha,\beta)$ denotes the $k\textsuperscript{th}$ moment of the generalized Hankel limiting spectral distribution, 
\begin{align}\label{eq:eqref322}
\begin{split}
M_4(H,\alpha,\beta)\ = \ 2,
\end{split}
\end{align}
the same as for original Hankel matrices.
%%%%%%%%%%%%%%%%%%%%%%%%%%%%%%%%%%%%%%%%%%%%%%%%%%%%%%%%%%%%%%%%%%%%%%%%%%%%%%%%%%%%%%%%%%%%%%%%%%%%%%%%%%%%%%%%%%%%%%%%%%%%%%
%%%%%%%%%%%%%%%%%%%%%%%%%%%%%%%%%%%%%%%%%%%%%%%%%%%%%%%%%%%%%%%%%%%%%%%%%%%%%%%%%%%%%%%%%%%%%%%%%%%%%%%%%%%%%%%%%%%%%%%%%%%%%%
\subsubsection{Sixth Moment}
By Lemma \ref{lem:lemma22}, the words \textit{aabbcc}, \textit{aabccb}, and other words isomorphic to them contribute one to the sixth moment.  Using the process of ``lifting'', the word \textit{aabcbc} contributes zero, since the non-contributing structure \textit{abab} is embedded within that word.  Counting linear constraints, the word \textit{abacbc} contributes zero when $\alpha \ne \beta$.  It also contributes zero when $\alpha = \beta$, according to calculations in \cite{B}.
For the word \textit{abcabc}, there are two sets of contributing L-value equations:
\begin{align}
\begin{split}
\beta\pi(0)+\alpha\pi(1) &= \alpha\pi(3)+\beta\pi(4)\\
\alpha\pi(1)+\beta\pi(2) &= \beta\pi(4)+\alpha\pi(5)\\
\beta\pi(2)+\alpha\pi(3) &= \alpha\pi(5)+\beta\pi(6)
\end{split}
\end{align}
and
\begin{align}
\begin{split}
\alpha\pi(0)+\beta\pi(1) &= \beta\pi(3)+\alpha\pi(4)\\
\beta\pi(1)+\alpha\pi(2) &= \alpha\pi(4)+\beta\pi(5)\\
\alpha\pi(2)+\beta\pi(3) &= \beta\pi(5)+\alpha\pi(6).
\end{split}
\end{align}
Following the same integration procedure that applied in \S\ref{sec:sec31}, it can be shown that
\begin{align}M_6(H,\alpha,\beta)\ = \ \begin{cases}\label{eq:eqref323}
       5 + \frac{\alpha}{2}\frac{\beta}{(\alpha + \beta)^2} & \alpha<\beta \\\\
     5 + \frac{\alpha}{2}\frac{\beta}{(\alpha+\beta)^2} & \alpha > \beta \\\\
     5\frac{1}{2} & \alpha\ = \ \beta\hspace{2mm}\text{  (original Hankel)}.
     \end{cases}
\end{align}
%%%%%%%%%%%%%%%%%%%%%%%%%%%%%%%%%%%%%%%%%%%%%%%%%%%%%%%%%%%%%%%%%%%%%%%%%%%%%%%%%%%%%%%%%%%%%%%%%%%%%%%%%%%%%%%%%%%%%%%%%%%%%%
%%%%%%%%%%%%%%%%%%%%%%%%%%%%%%%%%%%%%%%%%%%%%%%%%%%%%%%%%%%%%%%%%%%%%%%%%%%%%%%%%%%%%%%%%%%%%%%%%%%%%%%%%%%%%%%%%%%%%%%%%%%%%%
\subsubsection{Existence of Higher Moments, Bounds, and Limiting Behavior}
As in \S\ref{sec:sec31}, although we cannot find a closed-form expression for higher moments, all higher moments exist and are finite.
\begin{lem}[Existence of Higher Moments]\label{lem:lemma34} If the probability distribution $p$ has mean 0, variance 1, and finite higher moments, then for all $k$, $M_k = \lim_{N\rightarrow\infty}M_k(N)$ exists and is finite for generalized Hankel matrices.\end{lem}
\begin{proof}
The proof follows as in Lemma \ref{lem:lemma31}.
\end{proof}
\begin{lem}[Bounds on the Moments]\label{lem:lemma35}  Let $C_{2k}$ be the $2k\textsuperscript{th}$ moment of the semicircle distribution, $M_{2k}(H)$ the $2k\textsuperscript{th}$ moment of the Hankel ensemble limiting distribution, and $M_{2k}\left(H,\alpha,\beta\right)$ the $2k\textsuperscript{th}$ moment of the generalized Hankel ensemble limiting distribution.  Then $C_{2k}\le M_{2k}\left(H,\alpha,\beta\right) \le M_{2k}(H)$ for all nonnegative integers $k$.\end{lem}
\begin{proof}
The proof follows as in Lemma \ref{lem:lemma32}.
\end{proof}
In the limit as either $\alpha$ or $\beta$ becomes very large, and the other is fixed, the moments of the generalized Hankel ensemble approach those of the semicircle measure.\footnote{This is evidenced by Eq. \eqref{eq:eqref323}.}
\begin{lem}[Limiting Behavior]\label{lem36} For fixed $\alpha,\hspace{2mm}\text{lim}_{\beta \rightarrow \infty}  \hspace{2mm}M_{2k}(H,\alpha,\beta) = C_{2k}$ and for fixed $\beta,\hspace{2mm}\text{lim}_{\alpha \rightarrow \infty}  \hspace{2mm}M_{2k}(H,\alpha,\beta) = C_{2k}$  for generalized Hankel matrices when the limits are taken appropriately.\end{lem}
\begin{proof}
The proof follows as in Lemma \ref{lem:lemma33}.
\end{proof}
In working towards proving Theorem \ref{thm:thm14}, we have shown that the fourth and sixth moments for both ensembles are functions of $\alpha$ and $\beta$ and that in the limit as either parameter tends to infinity, the limiting distributions converge to the semicircle.  It is also evident that when these parameters are fixed, the limiting spectral distributions are non-semicircular.  
\begin{lem}[Non-semicircular Limiting Distributions]\label{lem:lemma37}  The limiting spectral distributions for generalized Toeplitz and generalized Hankel matrices are non-semicircular.  
\begin{proof}
For any value of $\alpha$ and $\beta$, $M_6(T,\alpha,\beta)$ and $M_6(H,\alpha,\beta$) are not equal to $C_6 = 5$.  The calculated moment values are shown in the table below.\\
\begin{center}
{\small\textsc{Table 1.} Fourth and Sixth Moments for Generalized Toeplitz and Hankel Matrices}
\resizebox{.8\columnwidth}{!}{
\begin{tabular}{cc}
\\\hline\hline
Generalized Toeplitz & Generalized Hankel \\[0.5ex]
\hline \\
$M_4(T,\alpha,\beta)$ = $\begin{cases}2 & \hspace{13mm}\alpha\ne\beta \\ 2\frac{2}{3} & \hspace{13mm}\alpha = \beta \\ \end{cases}$ & $M_4(H,\alpha,\beta)$ = $\begin{cases} 2 & \hspace{7mm}\text{for any } \alpha,\hspace{1mm}\beta \end{cases}$ \\\\
$M_6(T,\alpha,\beta)$ = $\begin{cases} 5 + \frac{\alpha}{2}\frac{1}{\alpha + \beta} & \alpha < \beta \\ 5 + \frac{\beta}{2}\frac{1}{\alpha + \beta} & \alpha > \beta \\ 11 & \alpha = \beta\end{cases}$ & $M_6(H,\alpha,\beta)$ = $\begin{cases} 5 + \frac{\alpha}{2}\frac{\beta}{(\alpha + \beta)^2} & \alpha < \beta \\ 5 + \frac{\alpha}{2}\frac{\beta}{(\alpha + \beta)^2} & \alpha > \beta \\ 5\frac{1}{2} & \alpha = \beta \end{cases}$\\\\
\hline
\end{tabular}
}    
\end{center}
\end{proof}
\end{lem}
In \S\ref{sec:sec5} we complete the proof of Theorem \ref{thm:thm14} by describing proofs of convergence in probability and almost sure convergence that apply to generalized Toeplitz and Hankel matrices.  
%%%%%%%%%%%%%%%%%%%%%%%%%%%%%%%%%%%%%%%%%%%%%%%%%%%%%%%%%%%%%%%%%%%%%%%%%%%%%%%%%%%%%%%%%%%%%%%%%%%%%%%%%%%%%%%%%%%%%%%%%%%%%%
%%%%%%%%%%%%%%%%%%%%%%%%%%%%%%%%%%%%%%%%%%%%%%%%%%%%%%%%%%%%%%%%%%%%%%%%%%%%%%%%%%%%%%%%%%%%%%%%%%%%%%%%%%%%%%%%%%%%%%%%%%%%%%
%%%%%%%%%%%%%%%%%%%%%%%%%%%%%%%%%%%%%%%%%%%%%%%%%%%%%%%%%%%%%%%%%%%%%%%%%%%%%%%%%%%%%%%%%%%%%%%%%%%%%%%%%%%%%%%%%%%%%%%%%%%%%%
%%%%%%%%%%%%%%%%%%%%%%%%%%%%%%%%%%%%%%%%%%%%%%%%%%%%%%%%%%%%%%%%%%%%%%%%%%%%%%%%%%%%%%%%%%%%%%%%%%%%%%%%%%%%%%%%%%%%%%%%%%%%%%
%%%%%%%%%%%%%%%%%%%%%%%%%%%%%%%%%%%%%%%%%%%%%%%%%%%%%%%%%%%%%%%%%%%%%%%%%%%%%%%%%%%%%%%%%%%%%%%%%%%%%%%%%%%%%%%%%%%%%%%%%%%%%%
%%%%%%%%%%%%%%%%%%%%%%%%%%%%%%%%%%%%%%%%%%%%%%%%%%%%%%%%%%%%%%%%%%%%%%%%%%%%%%%%%%%%%%%%%%%%%%%%%%%%%%%%%%%%%%%%%%%%%%%%%%%%%%
\section{Polynomial Toeplitz and Hankel Matrices}\label{sec:sec4}
In this section we work towards proving Theorem \ref{thm:thm15}.  In \S\ref{sec:sec41} we examine the polynomial Toeplitz ensemble in detail and prove that the empirical spectral distributions converge weakly to the semicircle distribution.  In \S\ref{sec:sec42} we show that the same result holds for polynomial Hankel matrices.  
%%%%%%%%%%%%%%%%%%%%%%%%%%%%%%%%%%%%%%%%%%%%%%%%%%%%%%%%%%%%%%%%%%%%%%%%%%%%%%%%%%%%%%%%%%%%%%%%%%%%%%%%%%%%%%%%%%%%%%%%%%%%%%
%%%%%%%%%%%%%%%%%%%%%%%%%%%%%%%%%%%%%%%%%%%%%%%%%%%%%%%%%%%%%%%%%%%%%%%%%%%%%%%%%%%%%%%%%%%%%%%%%%%%%%%%%%%%%%%%%%%%%%%%%%%%%%
\subsection{Polynomial Toeplitz Matrices}\label{sec:sec41}
Recall that if $p_1(x) = a_mx^m+a_{m-1}x^{m-1}+\cdots+a_0$ \text{ and }$p_2(x) = b_nx^n+b_{n-1}x^{n-1}+\cdots+b_0$ are polynomials with integer coefficients and $m\ne n$, the link function for the polynomial Toeplitz matrices is\footnote{Note that this link function is of the form Eq. \eqref{2.1}.  }
\begin{align}\label{4.1}L_{PT}(i,j) \ := \  \begin{cases}
p_1(i) - p_2(j) & i \le j\\
-p_2(i) + p_1(j) & i > j.
\end{cases}
\end{align}
A $5\times 5$ matrix with $p_1(x) = x^2$ and $p_2(x) = x$ has the structure
\begin{align} A_5= \left( \begin{array}{ccccc}
a_{0} & a_{-1}& a_{-2}& a_{-3} & a_{-4}\\
a_{-1} & a_{2}& a_{1}& a_0 & a_{-1}\\
a_{-2} &a_{1}& a_{6}& a_{5} & a_{4}\\
a_{-3} & a_{0} & a_{5}& a_{12}& a_{11} \\
a_{-4} & a_{-1}& a_{4}& a_{11} & a_{20} \end{array} \right). \end{align}\par
We now show that this link function yields a semicircular limiting spectral distribution.  First, we assume that the polynomials $p_1$ and $p_2$ are monotonic on $\mathbb{N}$.  We will remove this assumption later.
%%%%%%%%%%%%%%%%%%%%%%%%%%%%%%%%%%%%%%%%%%%%%%%%%%%%%%%%%%%%%%%%%%%%%%%%%%%%%%%%%%%%%%%%%%%%%%%%%%%%%%%%%%%%%%%%%%%%%%%%%%%%%%
%%%%%%%%%%%%%%%%%%%%%%%%%%%%%%%%%%%%%%%%%%%%%%%%%%%%%%%%%%%%%%%%%%%%%%%%%%%%%%%%%%%%%%%%%%%%%%%%%%%%%%%%%%%%%%%%%%%%%%%%%%%%%%
\subsubsection{Catalan Words}
By Lemma \ref{lem:lemma22}, every Catalan word of length $2k$ contributes one to the $2k\textsuperscript{th}$ moment.
%%%%%%%%%%%%%%%%%%%%%%%%%%%%%%%%%%%%%%%%%%%%%%%%%%%%%%%%%%%%%%%%%%%%%%%%%%%%%%%%%%%%%%%%%%%%%%%%%%%%%%%%%%%%%%%%%%%%%%%%%%%%%%
%%%%%%%%%%%%%%%%%%%%%%%%%%%%%%%%%%%%%%%%%%%%%%%%%%%%%%%%%%%%%%%%%%%%%%%%%%%%%%%%%%%%%%%%%%%%%%%%%%%%%%%%%%%%%%%%%%%%%%%%%%%%%%
\subsubsection{Crossed Words}
To prove that the limiting spectral measure is a semicircle, it suffices to show that all non-Catalan words contribute zero.  Specifically, we show that the contribution of a word that is fully crossed, or a word that has no adjacent pairs, is zero.  This shows that all non-Catalan words contribute zero, since the contribution of any non-Catalan word with adjacent pairs is calculated by adjacent ``lifting'', and any loss in degrees of freedom in lower moments will propagate through these adjacent pairs.
The argument follows as in Lemma \ref{lem:lemma33} by showing that for a specific non-adjacent pair in any fully crossed word, $a_{L(\pi(0),\pi(1))} = a_{L(\pi(2),\pi(3))}$, there are fewer than three degrees of freedom in choosing all four indices.  Without loss of generality, we assume that $a_{i_1i_2},\hspace{1mm}a_{i_3i_4}\in\text{ Zone 1}$ and $m>n$.  The relevant L-value equation can then be written as
\begin{equation}
p_1\left(\pi(0)\right) - p_1\left(\pi(2)\right)\ = \ p_2\left(\pi(1)\right) - p_2\left(\pi(3)\right).
\end{equation}
Although we have chosen the entries to be in specific zones, we relax these restrictions and assume that the only constraint is from the L-value equation.  First, assume that $\pi(1)\ge\pi(3)$, and let $b = \max(b_i)$ for $i\in\{0,1,\dots,n\}$.  Since indices are at least one and at most $N$, $\pi(0)$ and $\pi(2)$ must be chosen so that
\begin{equation}
0 \le p_1\left(\pi(0)\right) - p_1\left(\pi(2)\right) \le \sum_{i=0}^{n}b_{n-i}N^{n-i} - \sum_{i=0}^{n}b_{n-i}.
\end{equation}
We relax the upper bound and simply impose
\begin{equation}
0 \le p_1\left(\pi(0)\right) - p_1\left(\pi(2)\right) \le b(n+1)N^n.
\end{equation}
For large $N$, then, the number of valid choices for $\pi(0)$ and $\pi(2)$ becomes an integral:
\begin{equation}
\int_1^N\int_1^N\mathbb{I}\left(\pi(2)\le\pi(0)\text{ and }p_1\left(\pi(2)\right)\ge p_1\left(\pi(0)\right) - b(n+1)N^n\right)d_{\pi(2)}d_{\pi(0)}.
\end{equation}
We want to show that there are fewer than 2 degrees of freedom in choosing $\pi(0)$ and $\pi(2)$, as this implies $x<3$.  Let $\Pi$ be a placeholder for the integrand and let $C$ be a constant independent of $N$.  We then have
\begin{align}
\begin{split}
\int_1^N\int_1^N\Pi d_{\pi(2)}d_{\pi(0)} &= \int_1^C\int_1^N\Pi d_{\pi(2)}d_{\pi(0)} + \int_C^N\int_1^N\Pi d_{\pi(2)}d_{\pi(0)}\\
 &\le \int_1^C\int_0^N\Pi d_{\pi(2)}d_{\pi(0)} + \int_C^N\int_0^C\Pi d_{\pi(2)}d_{\pi(0)} + \int_C^N\int_C^N\Pi d_{\pi(2)}d_{\pi(0)}\\
& \le N(C-1) + C(N-C) + \int_C^N\int_C^N\Pi d_{\pi(2)}d_{\pi(0)}.
\end{split}
\end{align}
To check if there are two degrees of freedom, divide by $N^2$ and take the limit as $N$ tends to infinity; we can ignore the first two terms above.  For the third term, let $\tilde{p}_1(x) = \frac{p_1(x)}{a_m}$:
\begin{align}\label{4.9}
\begin{split}
 \int_C^N\int_C^N\Pi d_{\pi(2)}d_{\pi(0)}
 &\ = \ \int_C^N\int_C^{\pi(0)}\mathbb{I}\left(\tilde{p}_1\left(\pi(2)\right)\ge \tilde{p}_1\left(\pi(0)\right) - \frac{b(n+1)}{a_m}N^n\right) d_{\pi(2)}d_{\pi(0)}.\\
\end{split}
\end{align}
\begin{lem}[Polynomials Dominated by Leading Term]\label{lem:lemma41} Let $p(x) = x^k+a_{k-1}x^{k-1}+\cdots+a_0$.  We can choose $x$ large enough such that $(1-\epsilon)x^k<p(x)<(1+\epsilon)x^k$.\end{lem}
\begin{proof}
Let $q(x) = a_{k-1}x^{k-1}+\cdots+a_0$.  Then $\frac{p(x) - x^k}{x^k} = \frac{q(x)}{x^k}$.  Since $x^k$ dominates the leading order terms of $q(x)$ in the limit of large $x$, we can choose $x$ large enough such that for any $\epsilon'>0$,\hspace{1mm} $\frac{p(x) - x^k}{x^k} = \epsilon'$.  This gives us $p(x) = (1+\epsilon')x^k$.  If we let $\epsilon = 2 \epsilon'$, we have $(1-\epsilon)x^k<p(x)<(1+\epsilon)x^k$.
\end{proof}
By Lemma \ref{lem:lemma41}, for any $\epsilon$ there exists a $C$ large enough such that the integral in Eq. \eqref{4.9} is less than or equal to
\begin{align}
\begin{split}
{} & \int_C^N\int_C^{\pi(0)}\mathbb{I}\left((1+\epsilon)\pi(2)^m\ge(1-\epsilon)\pi(0)^m - \frac{b(n+1)}{a_m}N^n\right)d_{\pi(2)}d_{\pi(0)}\\
&\ = \ \int_C^N\int_C^{\pi(0)}\mathbb{I}\left(\pi(2)\ge\left((1-2\epsilon+O(\epsilon^2))\pi(0)^m - \frac{b(n+1)}{a_m(1+\epsilon)}N^n\right)^{\frac{1}{m}}\right)d_{\pi(2)}d_{\pi(0)}.
\end{split}
\end{align}
Split the outer integral at $\left(\frac{3b(n+1)N^n}{a_m(1+\epsilon)}\right)^{\frac{1}{m}}$.  The first term becomes negligible when we divide by $N^2$, and taking the limit as $N$ tends to infinity yields
\begin{align}
\begin{split}
{} & \int_C^{\left(\frac{3b(n+1)N^n}{a_m(1+\epsilon)}\right)^{\frac{1}{m}}}\int_C^{\pi(0)}\mathbb{I}\left(\pi(2)\ge\left((1-2\epsilon+O(\epsilon^2))\pi(0)^m - \frac{b(n+1)}{a_m(1+\epsilon)}N^n\right)^{\frac{1}{m}}\right)d_{\pi(2)}d_{\pi(0)}\hspace{90mm} \\
& \le \left(\frac{3b(n+1)N^n}{a_m(1+\epsilon)}\right)^{\frac{2}{m}}\frac{N^{\frac{2}{m}}}{2} - C\left(\frac{3b(n+1)N^n}{a_m(1+\epsilon)}\right)^{\frac{1}{m}}+\frac{C^2}{2},
\end{split}
\end{align}
since $\frac{n}{m}<1$.  For the second term, first pull out a factor of $\pi(0)^m$:
\begin{align}
\begin{split}
{} & \int_{\left(\frac{3b(n+1)N^n}{a_m(1+\epsilon)}\right)^{\frac{1}{m}}}^N\int_C^{\pi(0)}\mathbb{I}\left(\pi(2)\ge\left((1-2\epsilon+O(\epsilon^2))\pi(0)^m - \frac{b(n+1)}{a_m(1+\epsilon)}N^n\right)^{\frac{1}{m}}\right)d_{\pi(2)}d_{\pi(0)}\hspace{90mm} \\
& \le \int_{\left(\frac{3b(n+1)N^n}{a_m(1+\epsilon)}\right)^{\frac{1}{m}}}^N \pi(0)-\pi(0)\left(1-2\epsilon+O\left(\epsilon^2\right) - \frac{b(n+1)}{a_m(1+\epsilon)\pi(0)^m}N^n\right)^{\frac{1}{m}}d_{\pi(2)}d_{\pi(0)}.
\end{split}
\end{align}
From the lower bound of the integral, the second integrand term is at most\\ $\left(1-\frac{1}{3}-2\epsilon+O\left(\epsilon^2\right)\right)^{\frac{1}{m}}$.  We then apply the binomial expansion and get
\begin{align}
\begin{split}
\int_{\left(\frac{3b(n+1)N^n}{a_m(1+\epsilon)}\right)^{\frac{1}{m}}}^N \pi(0)-\pi(0)\left(1-\frac{2\epsilon}{m}+\frac{O\left(\epsilon^2\right)}{m} - \frac{b(n+1)}{ma_m(1+\epsilon)\pi(0)^m}N^n+\cdots\right)d_{\pi(2)}d_{\pi(0)}.
\end{split}
\end{align}
Since $m>n$ the terms that depend on $N^2$ are
\begin{align}
\begin{split}
\frac{N^2\epsilon}{m} + N^2O\left(\epsilon^2\right),
\end{split}
\end{align}
but we can choose $C$ large enough that $\epsilon\rightarrow 0$. Thus
\begin{align}
\begin{split}
\text{lim}_{N\rightarrow\infty}\frac{1}{N^2}\int_1^N\int_1^N\Pi d_{\pi(2)}d_{\pi(0)}\ = \ 0,
\end{split}
\end{align}
which implies that $\pi(0)$ and $\pi(2)$ are chosen with fewer than two degrees of freedom.  Choose $\pi(1)$, then $\pi(3)$ is fixed by the L-value equation, and there are fewer than three degrees of freedom in total.  Assuming $\pi(1)<\pi(3)$ similarly occurs with fewer than three degrees of freedom.  Choosing any other set of zones for the matched pair yields the same results, as will allowing $n>m$.  In all cases $x<3$.\par
Now we remove the assumption that $p_1$ and $p_2$ are monotonic.  Choose an upper bound on the absolute value of the roots of the first derivatives of the polynomials, $\frac{dp_1(x)}{dx}$ and $\frac{dp_2(x)}{dx}$.  One example, given by Rouch\'{e}'s Theorem\footnote{See Stein and Shakarchi \cite{SS}.}, is $R=\lceil\max(R_1,R_2)\rceil$, where
\begin{align}
\begin{split}
{} & R_1=1+\frac{1}{m|a_m|}\max(|a_1|,2|a_2|,\dots,(m-1)|a_{m-1}|)\\
{} & R_2=1+\frac{1}{n|b_n|}\max(|b_1|,2|b_2|,\dots,(n-1)|b_{n-1}|).
\end{split}
\end{align}
In counting the contribution to the even moments from fully crossed words, using Eq. \eqref{1.21}, we first count the number of terms where at least one entry $a_{i_s,i_{s+1}}$ is located in the upper-left $R\times R$ submatrix of entry locations.  It is in this region that $p_1$ and $p_2$ may not be monotonic, because the roots of the first derivatives are contained there.  For the $2k\textsuperscript{th}$ moment there are at most $k+1$ degrees of freedom.  One degree comes from choosing the first index arbitrarily, and the other $k$ degrees come from the $k$ choices of L-values for the $k$ matched pairs of entries.  When an entry is fixed in the $R\times R$ submatrix we lose a degree of freedom, because then there is a fixed number of choices independent of $N$ for a generating index to conform to that choice of L-value.  So, choose $N$ large enough and we only need to consider terms in which no entry appears in the upper-left $R\times R$ submatrix.  Since link function polynomials are monotonic outside of this submatrix, the results above apply to all polynomial link functions of the form Eq. \eqref{4.1}.
%%%%%%%%%%%%%%%%%%%%%%%%%%%%%%%%%%%%%%%%%%%%%%%%%%%%%%%%%%%%%%%%%%%%%%%%%%%%%%%%%%%%%%%%%%%%%%%%%%%%%%%%%%%%%%%%%%%%%%%%%%%%%%
%%%%%%%%%%%%%%%%%%%%%%%%%%%%%%%%%%%%%%%%%%%%%%%%%%%%%%%%%%%%%%%%%%%%%%%%%%%%%%%%%%%%%%%%%%%%%%%%%%%%%%%%%%%%%%%%%%%%%%%%%%%%%%
\subsection{Polynomial Hankel Matrices}\label{sec:sec42}
Recall that if $p_1(x) = a_mx^m+a_{m-1}x^{m-1}+\cdots+a_0$\text{ and }$ p_2(x) = b_nx^n+b_{n-1}x^{n-1}+\cdots+b_0$ are polynomials with integer coefficients and $m\ne n$, the link function for the polynomial Hankel matrices is\footnote{Note that this link function is of the form Eq. \eqref{2.1}.  }
\begin{align}L_{PH}(i,j) \ := \  \begin{cases}
p_1(i) + p_2(j) & i \le j\\
p_2(i) + p_1(j) & i > j.
\end{cases}
\end{align}
We show that this link function yields a semicircular limiting spectral distribution.  First, assume that the polynomials $p_1$ and $p_2$ are monotonic on $\mathbb{N}$.
%%%%%%%%%%%%%%%%%%%%%%%%%%%%%%%%%%%%%%%%%%%%%%%%%%%%%%%%%%%%%%%%%%%%%%%%%%%%%%%%%%%%%%%%%%%%%%%%%%%%%%%%%%%%%%%%%%%%%%%%%%%%%%
%%%%%%%%%%%%%%%%%%%%%%%%%%%%%%%%%%%%%%%%%%%%%%%%%%%%%%%%%%%%%%%%%%%%%%%%%%%%%%%%%%%%%%%%%%%%%%%%%%%%%%%%%%%%%%%%%%%%%%%%%%%%%%
\subsubsection{Catalan Words}
By Lemma \ref{lem:lemma22}, every Catalan word of length $2k$ contributes one to the $2k\textsuperscript{th}$ moment.
%%%%%%%%%%%%%%%%%%%%%%%%%%%%%%%%%%%%%%%%%%%%%%%%%%%%%%%%%%%%%%%%%%%%%%%%%%%%%%%%%%%%%%%%%%%%%%%%%%%%%%%%%%%%%%%%%%%%%%%%%%%%%%
%%%%%%%%%%%%%%%%%%%%%%%%%%%%%%%%%%%%%%%%%%%%%%%%%%%%%%%%%%%%%%%%%%%%%%%%%%%%%%%%%%%%%%%%%%%%%%%%%%%%%%%%%%%%%%%%%%%%%%%%%%%%%%
\subsubsection{Crossed Words}
To prove that the limiting spectral measure is a semicircle, it suffices to show that all non-Catalan words contribute zero.  Following the argument in \S\ref{sec:sec41}, pick matched entries appropriately and count the number of solutions to show that $x<3$.  Without loss of generality, assume that $a_{i_1i_2},\hspace{1mm}a_{i_3i_4}\in\text{ Zone 1}$ and $m>n$.  The relevant L-value equation can be written as
\begin{equation}
p_1\left(\pi(0)\right) - p_1\left(\pi(2)\right)\ = \ p_2\left(\pi(3)\right) - p_2\left(\pi(1)\right).
\end{equation}
At this point, the methods from the previous section apply.  In \S\ref{sec:sec5} we complete the proof of Theorem\ref{thm:thm15} by describing proofs of convergence in probability and almost sure convergence that apply to polynomial Toeplitz and Hankel matrices.   
%%%%%%%%%%%%%%%%%%%%%%%%%%%%%%%%%%%%%%%%%%%%%%%%%%%%%%%%%%%%%%%%%%%%%%%%%%%%%%%%%%%%%%%%%%%%%%%%%%%%%%%%%%%%%%%%%%%%%%%%%%%%%%
%%%%%%%%%%%%%%%%%%%%%%%%%%%%%%%%%%%%%%%%%%%%%%%%%%%%%%%%%%%%%%%%%%%%%%%%%%%%%%%%%%%%%%%%%%%%%%%%%%%%%%%%%%%%%%%%%%%%%%%%%%%%%%
%%%%%%%%%%%%%%%%%%%%%%%%%%%%%%%%%%%%%%%%%%%%%%%%%%%%%%%%%%%%%%%%%%%%%%%%%%%%%%%%%%%%%%%%%%%%%%%%%%%%%%%%%%%%%%%%%%%%%%%%%%%%%%
%%%%%%%%%%%%%%%%%%%%%%%%%%%%%%%%%%%%%%%%%%%%%%%%%%%%%%%%%%%%%%%%%%%%%%%%%%%%%%%%%%%%%%%%%%%%%%%%%%%%%%%%%%%%%%%%%%%%%%%%%%%%%%
%%%%%%%%%%%%%%%%%%%%%%%%%%%%%%%%%%%%%%%%%%%%%%%%%%%%%%%%%%%%%%%%%%%%%%%%%%%%%%%%%%%%%%%%%%%%%%%%%%%%%%%%%%%%%%%%%%%%%%%%%%%%%%
%%%%%%%%%%%%%%%%%%%%%%%%%%%%%%%%%%%%%%%%%%%%%%%%%%%%%%%%%%%%%%%%%%%%%%%%%%%%%%%%%%%%%%%%%%%%%%%%%%%%%%%%%%%%%%%%%%%%%%%%%%%%%%
\section{Convergence}\label{sec:sec5}
For each of the ensembles in this paper, all higher moments exist and are finite and the link functions satisfy Property B.  As a result, as $N\rightarrow \infty$ the empirical measures for matrices in each ensemble converge in probability and almost surely to a unique and universal limiting spectral distribution.\footnote{See \cite{B} for proof of the fact that the limiting distribution is universal for all matrix ensembles that satisfy Property B.}  The following sections briefly summarize general proofs of convergence for broad classes of random matrix ensembles, which completes the proof of the convergence claims in Theorem \ref{thm:thm14} and Theorem \ref{thm:thm15}.
%%%%%%%%%%%%%%%%%%%%%%%%%%%%%%%%%%%%%%%%%%%%%%%%%%%%%%%%%%%%%%%%%%%%%%%%%%%%%%%%%%%%%%%%%%%%%%%%%%%%%%%%%%%%%%%%%%%%%%%%%%%%%%
%%%%%%%%%%%%%%%%%%%%%%%%%%%%%%%%%%%%%%%%%%%%%%%%%%%%%%%%%%%%%%%%%%%%%%%%%%%%%%%%%%%%%%%%%%%%%%%%%%%%%%%%%%%%%%%%%%%%%%%%%%%%%%
\subsection{Existence and Uniqueness of Limiting Spectral Distributions}Arguments in \cite{B} prove that if a random matrix ensemble has a link function that satisfies Property B and the limiting moments exist, then the limiting spectral distribution exists.  Moreover, the limiting spectral distribution is uniquely specified by its moments.  Essentially, \cite{B} shows that Property B requires $M_{2k}(N) \le \frac{(2k)!}{2^kk!}\Delta(L)^k + O\left(\frac{1}{N}\right)$.  As $N\rightarrow\infty$, $M_{2k}\le \frac{(2k)!}{2^kk!}\Delta(L)^k$, which satisfies Riesz's condition.  By Theorem \ref{thm:thm12}, the limiting spectral distribution of the ensemble exists and is uniquely determined.
%%%%%%%%%%%%%%%%%%%%%%%%%%%%%%%%%%%%%%%%%%%%%%%%%%%%%%%%%%%%%%%%%%%%%%%%%%%%%%%%%%%%%%%%%%%%%%%%%%%%%%%%%%%%%%%%%%%%%%%%%%%%%%
%%%%%%%%%%%%%%%%%%%%%%%%%%%%%%%%%%%%%%%%%%%%%%%%%%%%%%%%%%%%%%%%%%%%%%%%%%%%%%%%%%%%%%%%%%%%%%%%%%%%%%%%%%%%%%%%%%%%%%%%%%%%%%
\subsection{Convergence in Probability} For convergence in probability, assume that all moments $M_{k}$ exist, are finite, and uniquely determine a probability distribution.  The empirical spectral distributions converge in probability to the limiting spectral distribution if the empirical moments converge in probability to the limiting moments.  Arguments in \cite{HM} for Toeplitz matrices show that by applying the triangle inequality and Chebyshev's Inequality to Eq. \eqref{eq:eq119}, it suffices to prove that for all nonnegative integers $k$,
\begin{align}\label{eq:eq51}
\begin{split}
\text{lim}_{N\rightarrow\infty}\left(\mathbb{E}\left[M_{k}\left(A_N\right)^2\right] - \mathbb{E}\left[M_{k}\left(A_N\right)\right]^2\right)\ = \ 0.
\end{split}
\end{align}
For our matrix ensembles, all remaining steps of the proof that count the contributions in Eq. \eqref{eq:eq51} follow trivially except changes in the constants $O_{k}\left(\frac{1}{N}\right)$, which do not alter the results.
%%%%%%%%%%%%%%%%%%%%%%%%%%%%%%%%%%%%%%%%%%%%%%%%%%%%%%%%%%%%%%%%%%%%%%%%%%%%%%%%%%%%%%%%%%%%%%%%%%%%%%%%%%%%%%%%%%%%%%%%%%%%%%
%%%%%%%%%%%%%%%%%%%%%%%%%%%%%%%%%%%%%%%%%%%%%%%%%%%%%%%%%%%%%%%%%%%%%%%%%%%%%%%%%%%%%%%%%%%%%%%%%%%%%%%%%%%%%%%%%%%%%%%%%%%%%%
\subsection{Almost Sure Convergence}For almost sure convergence, assume that all moments $M_k$ exist, are finite, and uniquely determine a probability distribution.  The empirical spectral distributions converge almost surely to the limiting spectral distribution if the empirical moments converge almost surely to the limiting moments.  The arguments in \cite{HM} designed for Toeplitz matrices again show that by applying triangle inequality and Chebyshev's Inequality to Eq. \eqref{eq:eq118}, it suffices to prove that for every nonnegative integer $k$,
\begin{align}
\begin{split}
\lim_{N\rightarrow\infty}\mathbb{E}[|M_k(A_N)-\mathbb{E}[M_k(A_N)]|^4]=O\left(\frac{1}{N^2}\right).
\end{split}
\end{align}
Then, by using combinatorics and the Borel-Cantelli Lemma, it can be shown that $M_k(A_N)\rightarrow M_k$ outside of a set of measure zero.  Again, for our matrix ensembles, all of the steps of the proof follow trivially except changes in the constants $O_k\left(\frac{1}{N^2}\right)$, which do not alter the arguments.
%%%%%%%%%%%%%%%%%%%%%%%%%%%%%%%%%%%%%%%%%%%%%%%%%%%%%%%%%%%%%%%%%%%%%%%%%%%%%%%%%%%%%%%%%%%%%%%%%%%%%%%%%%%%%%%%%%%%%%%%%%%%%%
%%%%%%%%%%%%%%%%%%%%%%%%%%%%%%%%%%%%%%%%%%%%%%%%%%%%%%%%%%%%%%%%%%%%%%%%%%%%%%%%%%%%%%%%%%%%%%%%%%%%%%%%%%%%%%%%%%%%%%%%%%%%%%
%%%%%%%%%%%%%%%%%%%%%%%%%%%%%%%%%%%%%%%%%%%%%%%%%%%%%%%%%%%%%%%%%%%%%%%%%%%%%%%%%%%%%%%%%%%%%%%%%%%%%%%%%%%%%%%%%%%%%%%%%%%%%%
%%%%%%%%%%%%%%%%%%%%%%%%%%%%%%%%%%%%%%%%%%%%%%%%%%%%%%%%%%%%%%%%%%%%%%%%%%%%%%%%%%%%%%%%%%%%%%%%%%%%%%%%%%%%%%%%%%%%%%%%%%%%%%
%%%%%%%%%%%%%%%%%%%%%%%%%%%%%%%%%%%%%%%%%%%%%%%%%%%%%%%%%%%%%%%%%%%%%%%%%%%%%%%%%%%%%%%%%%%%%%%%%%%%%%%%%%%%%%%%%%%%%%%%%%%%%%
%%%%%%%%%%%%%%%%%%%%%%%%%%%%%%%%%%%%%%%%%%%%%%%%%%%%%%%%%%%%%%%%%%%%%%%%%%%%%%%%%%%%%%%%%%%%%%%%%%%%%%%%%%%%%%%%%%%%%%%%%%%%%%
\section{Future Work: Other Polynomial Link Functions}\label{sec:sec6}
What can be said about the limiting spectral distribution for other polynomial link functions of the form Eq. \eqref{2.1}?  For example, let $\alpha$,$\hspace{1mm}\beta$ and $n$ be positive integers, and
 \begin{align} L_{\alpha,\beta}(i,j)\ = \ \begin{cases}
       \alpha i^n - \beta j^n  & i \le j \\
       -\beta i^n + \alpha j^n & i > j.
     \end{cases}
\end{align}
Using the methods described in this paper, we find that for $\alpha = \beta$ and $n=2$,
\begin{align}
\begin{split}
M_4\ = \ 2+\frac{8-\pi+2\text{Log}(4)}{12} < 2\frac{2}{3}.
\end{split}
\end{align}
Evidently, raising the variables to a higher power reduced the value of the fourth moment compared to the original real symmetric Toeplitz ensemble.  There are many other types of bivariate polynomial link functions to be explored, and there is potential for new and interesting limiting distributions to arise.
%%%%%%%%%%%%%%%%%%%%%%%%%%%%%%%%%%%%%%%%%%%%%%%%%%%%%%%%%%%%%%%%%%%%%%%%%%%%%%%%%%%%%%%%%%%%%%%%%%%%%%%%%%%%%%%%%%%%%%%%%%%%%%
%%%%%%%%%%%%%%%%%%%%%%%%%%%%%%%%%%%%%%%%%%%%%%%%%%%%%%%%%%%%%%%%%%%%%%%%%%%%%%%%%%%%%%%%%%%%%%%%%%%%%%%%%%%%%%%%%%%%%%%%%%%%%%
%%%%%%%%%%%%%%%%%%%%%%%%%%%%%%%%%%%%%%%%%%%%%%%%%%%%%%%%%%%%%%%%%%%%%%%%%%%%%%%%%%%%%%%%%%%%%%%%%%%%%%%%%%%%%%%%%%%%%%%%%%%%%%
%%%%%%%%%%%%%%%%%%%%%%%%%%%%%%%%%%%%%%%%%%%%%%%%%%%%%%%%%%%%%%%%%%%%%%%%%%%%%%%%%%%%%%%%%%%%%%%%%%%%%%%%%%%%%%%%%%%%%%%%%%%%%%
%%%%%%%%%%%%%%%%%%%%%%%%%%%%%%%%%%%%%%%%%%%%%%%%%%%%%%%%%%%%%%%%%%%%%%%%%%%%%%%%%%%%%%%%%%%%%%%%%%%%%%%%%%%%%%%%%%%%%%%%%%%%%%
%%%%%%%%%%%%%%%%%%%%%%%%%%%%%%%%%%%%%%%%%%%%%%%%%%%%%%%%%%%%%%%%%%%%%%%%%%%%%%%%%%%%%%%%%%%%%%%%%%%%%%%%%%%%%%%%%%%%%%%%%%%%%%
\appendix
%%%%%%%%%%%%%%%%%%%%%%%%%%%%%%%%%%%%%%%%%%%%%%%%%%%%%%%%%%%%%%%%%%%%%%%%%%%%%%%%%%%%%%%%%%%%%%%%%%%%%%%%%%%%%%%%%%%%%%%%%%%%%%
%%%%%%%%%%%%%%%%%%%%%%%%%%%%%%%%%%%%%%%%%%%%%%%%%%%%%%%%%%%%%%%%%%%%%%%%%%%%%%%%%%%%%%%%%%%%%%%%%%%%%%%%%%%%%%%%%%%%%%%%%%%%%%
\section{Numerical Methods}\label{sec:sec7}
Numerical methods were an invaluable tool that illuminated moment contributions and helped to guide all of the arguments in this paper.  For each link function, we studied simulations of the moment values and histograms of the normalized eigenvalues.  For each moment and for a fixed value of $N$ we used the Eigenvalue-Trace Lemma to calculate the moment of the eigenvalue distribution for a particular matrix and then averaged over a large number of such random matrices to get an approximation for the average limiting moment.   \par
In the four tables below we present the data from simulating 1,000 real symmetric $2000 \times 2000$ matrices with examples of the generalized Toeplitz, generalized Hankel, polynomial Toeplitz, and polynomial Hankel link functions.  While this method was very accurate for low moments, for higher moments the big-Oh constants grow quite large and make it computationally difficult to simulate a representative sample of sufficiently large matrices.\\\\
{\small\textsc{Table 2.} Generalized Toeplitz Sixth Moment}\hspace{5.00mm}{\small\textsc{Table 3.} Generalized Hankel Sixth Moment}
\begin{center}
\begin{tabular}{cc}
\centering
\resizebox{.5\columnwidth}{!}{
\begin{tabular}{ccccc}
\hline\hline
$\alpha$ & $\beta$ & Predicted & Observed & Observed/Predicted \\[0.5ex]
\hline \\
1 & 1 & 11.000 & 11.025 & 1.002  \\\\
1 & 2 & 5.167 & 5.077 & 0.983 \\\\
1 & 3 & 5.125 & 5.059 & 0.987 \\\\
1 & 4 & 5.100 & 5.021 & 0.985 \\\\
1 & 5 & 5.083 & 5.030 & 0.990 \\\\
\hline
\end{tabular}
}
&
\resizebox{.5\columnwidth}{!}{
\begin{tabular}{ccccc}
\hline\hline
$\alpha$ & $\beta$ & Predicted & Observed & Observed/Predicted \\[0.5ex]
\hline \\
1 & 1 & 5.500 & 5.500 & 1.000 \\\\
1 & 2 & 5.111 & 5.069 & 0.992 \\\\
1 & 3 & 5.094 & 5.048 & 0.991  \\\\
1 & 4 & 5.080 & 5.017 & 0.988 \\\\
1 & 5 & 5.069 & 5.028 & 1.008 \\\\
\hline
\end{tabular}
}
\end{tabular}
\end{center}
\vspace{5mm}
\hspace{4.5mm}{\small\textsc{Table 4.} Polynomial Toeplitz} \hspace{30mm} {\small\textsc{Table 5.} Polynomial Hankel}
\begin{center}
\begin{tabular}{cc}
\centering
\resizebox{.5\columnwidth}{!}{
\begin{tabular}{cccc}
\hline\hline
Moment & Predicted & Observed & Observed/Predicted \\[0.5ex]
\hline \\
$2$ & 1.000 & 1.000 & 1.000 \\\\
$4$ & 2.000 & 2.000 & 1.000 \\\\
$6$ & 5.000 & 5.006 & 1.001 \\\\
$8$ & 14.000 & 14.014 & 1.001 \\\\
$10$ & 42.000 & 42.086 & 1.002 \\\\
\hline
\end{tabular}
}
&
\resizebox{.5\columnwidth}{!}{
\begin{tabular}{cccc}
\hline\hline
Moment & Predicted & Observed & Observed/Predicted \\[0.5ex]
\hline \\
$2$ & 1.000 & 1.000 & 1.000 \\\\
$4$ & 2.000 & 2.000 & 1.000 \\\\
$6$ & 5.000 & 5.004 & 1.001 \\\\
$8$ & 14.000 & 14.02 & 1.001 \\\\
$10$ & 42.000 & 42.078 & 1.002 \\\\
\hline
\end{tabular}
}
\end{tabular}
\end{center}
%%%%%%%%%%%%%%%%%%%%%%%%%%%%%%%%%%%%%%%%%%%%%%%%%%%%%%%%%%%%%%%%%%%%%%%%%%%%%%%%%%%%%%%%%%%%%%%%%%%%%%%%%%%%%%%%%%%%%%%%%%%%%%
%%%%%%%%%%%%%%%%%%%%%%%%%%%%%%%%%%%%%%%%%%%%%%%%%%%%%%%%%%%%%%%%%%%%%%%%%%%%%%%%%%%%%%%%%%%%%%%%%%%%%%%%%%%%%%%%%%%%%%%%%%%%%
%%%%%%%%%%%%%%%%%%%%%%%%%%%%%%%%%%%%%%%%%%%%%%%%%%%%%%%%%%%%%%%%%%%%%%%%%%%%%%%%%%%%%%%%%%%%%%%%%%%%%%%%%%%%%%%%%%%%%%%%%%%%%%

\end{document}